\documentclass[a4paper,reqno,11pt,
]{amsart}
\usepackage{
amssymb,
amsmath,
amsthm,
eucal,
dsfont,
multicol,
mathrsfs,color,
hyperref
}

\makeatletter\renewcommand*{\eqref}[1]{%
\hyperref[{#1}]{\textup{\tagform@{\!\!\ref*{#1}}}}%
}\makeatother %hide when putting this file to arXiv

\makeatletter
 
  \@addtoreset{equation}{section}
 \makeatother
\theoremstyle{plain}

\newtheorem{theorem}{Theorem}[section]
\newtheorem{lemma}[theorem]{Lemma}
\newtheorem{proposition}[theorem]{Proposition}

\theoremstyle{definition}

\newtheorem{remark}[theorem]{Remark}

\newtheorem{assumption}{Assumption}

\newcommand{\bignorm}[1]{{\left\|#1\right\|}}
\newcommand{\norm}[1]{{\|#1\|}}

\def\Re{\mathop{\mathrm{Re}}\nolimits}
\def\Im{\mathop{\mathrm{Im}}\nolimits}
\def\loc{\mathop{\mathrm{loc}}\nolimits}

\def\ap{\mathop{\mathrm{ap}}\nolimits}
\def\R{{\mathbb{R}}}

\def\C{{\mathbb{C}}}

\def\F{{\mathcal{F}}}
\def\H{{\mathcal{H}}}

\def\L{{\mathcal{L}}}

\def\D{{\mathcal{D}}}

\def\M{{\mathcal{M}}}

\def\<{{\langle}}
\def\>{{\rangle}}

\def\ep{{\varepsilon}}
\def\ds{\displaystyle}
\DeclareMathOperator*{\slim}{s-lim}
\DeclareMathOperator*{\wlim}{w-lim}

\title[Modified scattering for inhomogeneous NLS]{Modified scattering for inhomogeneous nonlinear Schr\"odinger equations with and without inverse-square potential}

\author[K. Aoki]{Kazuki Aoki}
\address[K. Aoki]{}
\email{k.aoki5296@gmail.com}
\author[T. Inui]{Takahisa Inui}
\address[T. Inui]{Department of Mathematics, Graduate School of Science, Osaka University, Toyonaka, Osaka 560-0043, Japan}
\email{inui@math.sci.osaka-u.ac.jp}
\author[H. Miyazaki]{Hayato Miyazaki}
\address[H. Miyazaki]{Teacher Training Courses, Faculty of Education, Kagawa University, Takamatsu, Kagawa 760-8522, Japan}
\email{miyazaki.hayato@kagawa-u.ac.jp}
\author[H. Mizutani]{Haruya Mizutani}
\address[H. Mizutani]{Department of Mathematics, Graduate School of Science, Osaka University, Toyonaka, Osaka 560-0043, Japan}
\email{haruya@math.sci.osaka-u.ac.jp}
\author[K. Uriya]{Kota Uriya}
\address[K. Uriya]{Department of Applied Mathematics, Faculty of Science, Okayama University of Science,  Okayama, Okayama 700-0005, Japan}
\email{uriya@xmath.ous.ac.jp}

\begin{document}
%\date{\today}

\begin{abstract}
We consider the final state problem for the inhomogeneous nonlinear Schr\"odinger equation with a critical long-range nonlinearity. Given a prescribed asymptotic profile, which has a logarithmic phase correction compared with the free evolution, we construct a unique global solution which converges to the profile. As a consequence, the existence of modified wave operators for localized small scattering data is obtained. We also study the same problem for the case with the critical inverse-square potential under the radial symmetry. In particular, we construct the modified wave operators for the long-range nonlinear Schr\"odinger equation with the critical inverse-square potential in three space dimensions, under the radial symmetry. 
\end{abstract}
\maketitle


\section{Introduction}
\subsection{Background and problems}
In this paper we mainly consider the inhomogeneous nonlinear Schr\"odinger equation (NLS): \begin{align}
\label{INLS}
i\partial_tu+\Delta u=N(u),\quad(t,x)\in \R\times \R^d,
\end{align}
where $u=u(t,x)$ is a $\C$-valued unknown function, $\lambda\in \R$ and
$$ 
N(u)=\lambda |x|^{-\beta}|u|^{\alpha}u.
$$
The inhomogeneous NLS can be regarded as one of generalizations of the standard NLS with the constant coefficient nonlinearity $\lambda |u|^{\alpha}u$: 
\begin{align}
\label{NLS_1}
i\partial_tu+\Delta u=\lambda |u|^{\alpha}u,\quad \alpha>0,
\end{align}
which is a mathematical model for the wave propagation in nonlinear optics. Variable coefficient nonlinearities such as $N(u)$ naturally appear when one takes the inhomogeneity of the medium into account (see e.g. \cite[Chapter 6]{LiKe}). 

Recently, the inhomogeneous NLS \eqref{INLS} has attracted several attention in the mathematical literature. From a mathematical viewpoint, an interest of the model \eqref{INLS} is that it has a similar scaling symmetry \eqref{scaling} as that for \eqref{NLS_1} thanks to the homogeneity of the term $|x|^{-\beta}$. One could thus discuss various properties of the solutions to \eqref{INLS} based on scaling consideration as in case of \eqref{NLS_1}, which is one of motivation to study \eqref{INLS}. Another motivation, which might be not widely known, is a unitary equivalence (under the radial symmetry) between \eqref{INLS} with $(d,\alpha,\beta)=(2,2/3,1/3)$ and the following NLS with the inverse-square potential:
\begin{align}
\label{NLSI0}
i\partial_t\tilde u+\Delta \tilde u+\frac{1}{4|x|^2}\tilde u=\tilde \lambda |\tilde u|^{2/3}\tilde u,\quad (t,x)\in \R\times \R^3.
\end{align}
The NLS with an inverse-square potential $a|x|^{-2}$ has also attracted a lot of attention in the past decade (see e.g. \cite{BPST,ZhZh,KMVZZ}). As an application of our results for \eqref{INLS}, we also study \eqref{NLSI0} as well as a more general model \eqref{NLSI} than \eqref{NLSI0} in this paper.

 As mentioned above, \eqref{INLS} is invariant under the scaling
\begin{align}
\label{scaling}
u(t,x)\mapsto u_\lambda(t,x)=\lambda^{\frac{2-\beta}{\alpha}}u(\lambda^2t,\lambda x),
\end{align}
namely if $u$ solves \eqref{INLS}, then so does $u_\lambda$. Hence, the  natural scaling-critical Sobolev space is
$$
\dot H^{s_c}(\R^d),\quad s_c=\frac d2-\frac{2-\beta}{\alpha}. 
$$
The basic well-posedness problem in $L^2$ or $H^1$ for \eqref{INLS} has been studied by \cite{GeSt,Guzman} (see also \cite{Dinh} for more recent developments). The problem of scattering for \eqref{INLS}  in case of $0\le s_c\le 1$ has been also studied by many authors  (see e.g. \cite{FaGu_JDE, Dinh, FaGu_BBMC, MMZ} and reference therein), where it has been shown that the global solution $u(t)$ scatters to a solution to the free Schr\"odinger equation (called a free solution) in the sense that there exist $u_\pm$ such that $$u(t)=e^{it\Delta}u_\pm+o(1),\quad t\to\pm \infty,$$ in a suitable Sobolev space, under some appropriate conditions on the parameters $\alpha,\beta,\lambda,d$ and the initial data. Such a phenomenon is called the {\it short-range scattering} or simply {\it scattering} and, in such a case, $N(u)$ is said to be of {\it short-range} type. However, the problem of scattering is less understood  for the case $s_c<0$, or equivalently $0<\alpha d/2+\beta<2$. In particular, there is no previous literature on the global dynamics of the solutions to \eqref{INLS} if $0<\alpha d/2+\beta\le1$ in which case we call $N(u)$ of {\it long-range} type. 

On the other hand, the scattering theory for the standard NLS \eqref{NLS_1} has been extensively studied in both of the short-range and long-range cases. It is well known that the exponent $\alpha=2/d$ is critical in the following sense (see e.g. \cite{Cazenave}):
\begin{itemize}
\item If $2/d<\alpha\le 4/(d-2)$ for $d\ge3$ or $2/d<\alpha<\infty$ for $d=1,2$, then any sufficiently smooth, localized and small solution of \eqref{NLS_1} scatters to a free solution. 
\item If $\alpha \le 2/d$, then there is no non-trivial solution which scatters to a free solution. 
\end{itemize}
%For the scattering theory in the short-range case $\alpha>2/d$, we also refer to monographs \cite{Cazenave,Tao,Dodson}. 
Moreover, the asymptotic behavior of the solution has been also studied in the critical case $\alpha =2/d$. The final state problem for \eqref{NLS_1} with $\alpha =2/d$ was studied by Ozawa \cite{Ozawa1991} for $d=1$ and Ginibre--Ozawa \cite{GiOz} for $d=2,3$, where they showed the existence of the unique global solution with a given scattering datum $\varphi$ and the prescribed asymptotic behavior as $t\to \infty$: 
\begin{align}
\label{NLS_2}
u(t)=\frac{1}{(2it)^{d/2}}\varphi\left(\frac{x}{2t}\right)\exp\left(i\frac{|x|^2}{4t}-i\frac\lambda2\left|\varphi\left(\frac{x}{2t}\right)\right|^{2/d}\log t\right)+o(1),\quad t\to\infty.
\end{align}
Note that \eqref{NLS_2} has a logarithmic phase correction term compared with one for the free solution 
\begin{align}
\label{free_asymptotics}
e^{it\Delta}\F^{-1}\varphi=\frac{1}{(2it)^{d/2}}\varphi\left(\frac{x}{2t}\right)\exp\left(i\frac{|x|^2}{4t}\right)+o(1),\quad t\to \infty.
\end{align}
This type of phenomena is called the {\it modified scattering}. The modified scattering for the Cauchy problem of \eqref{NLS_1} with $\alpha=2/d$ and $1\le d\le 3$ was studied by Hayashi--Naumkin \cite{HaNa1998}, namely they determined the asymptotic behavior of the (small) solution to the associated Cauchy problem, which is essentially given by \eqref{NLS_2} with some $\varphi$ determined by the initial state. For more literatures and recent developments on the modified scattering for NLS, we refer to \cite{HWN,MaMi,MMU}. 
 
Motivated by these observations, we initiate the study of the long-range scattering theory for the inhomogeneous NLS \eqref{INLS}. We focus on the final state problem and prove the modified scattering for \eqref{INLS} with the parameters $\alpha,\beta$ satisfying the following conditions:
\begin{align}
\label{long_range_INLS}
0<\beta<\min\Big\{\frac d2,1\Big\},\quad\frac d2+\frac\beta\alpha<\alpha+1,\quad \frac{\alpha d}{2}+\beta=1.
\end{align}
Note that \eqref{long_range_INLS} holds if and only if
\begin{align}
\label{remark_1_1}
\max\Big\{\frac1d,\frac{\sqrt5-1}{2}\Big\}<\alpha\le \frac2d,\quad 0<\beta<\beta_d:=\min\Big\{\frac d2,\frac{d+4-d\sqrt5}{4}\Big\},\quad \frac{\alpha d}{2}+\beta=1,
\end{align}
where $\beta_d$ satisfies, for instance, $\beta_1=\frac12$, $\beta_2\in (\frac13,\frac12)$, $\beta_3\in (\frac{1}{14},\frac{1}{13})$ and $\beta_d<0$ for $d\ge4$. In particular, \eqref{long_range_INLS} makes sense if and only if $1\le d\le 3$. 

The condition $\alpha d/2+\beta=1$ is most important among \eqref{long_range_INLS} and the threshold from the viewpoint of short-range scattering vs. long-range scattering dichotomy. Indeed, as a supplementary result, we also prove the scattering to a free solution for the case $1<{\alpha d}/{2}+\beta\le2$ and $\lambda>0$. Here we give two rough explanations where this condition comes from. First, it has been observed for many dispersive equations that the finiteness/infiniteness of the quantity 
$$
\int_1^\infty \norm{N(\cdot,u_{\mathrm{free}}(t,\cdot))}_{L^2}dt
$$
 for the perturbation term $N(x,u)$ is closely related to the asymptotic behaviors of the solutions, where $u_{\mathrm{free}}$ is a free solution to the corresponding linear equation without perturbations. Namely, one can heuristically expect that a sufficiently small, localized and smooth solution scatters to a free solution if the quantity is finite, otherwise the modified scattering has to be taken into account. In case of \eqref{INLS}, thanks to  \eqref{free_asymptotics}, the above quantity can be replaced by
$$
\int_1^\infty |2t|^{-\alpha d /2}\bignorm{|x|^{-\beta}\left|\varphi\left(\frac{x}{2t}\right)\right|^\alpha}_{L^\infty} dt=\norm{|x|^{-\beta/\alpha}\varphi}_{L^\infty}^\alpha\int_1^\infty |2t|^{-\alpha d /2-\beta}dt
$$ 
which leads the condition $ \alpha d/2+\beta=1$ should be critical. Second, \eqref{INLS} is regarded as an intermediate model between the two cases $\beta=0$ and $\alpha=0$. When $ \alpha d/2+\beta=1$, the former case is the NLS \eqref{NLS_1} with $\alpha=2/d$ for which one has the modified scattering as mentioned above, while the latter is the linear Schr\"odinger equation with the Coulomb potential: 
\begin{align}
\label{Coulomb}
i\partial_tu+\Delta u=\lambda |x|^{-1}u.
\end{align}
It is well known (see e.g. \cite{DeGe}) that, as with the former case, \eqref{Coulomb} is also of long-range type so that the asymptotic behavior of the continuous part of the solution is given by
\begin{align}
\label{Coulomb_2}
P_cu(t)=\frac{1}{(2it)^{d/2}}\varphi\left(\frac{x}{2t}\right)\exp\left(i\frac{|x|^2}{4t}-i\frac\lambda2\left|\frac{x}{2t}\right|^{-1}\log t\right)+o(1),\quad t\to \infty
\end{align}
with some scattering datum $\varphi$. 
We will in fact show the existence of the global solution to \eqref{INLS} whose asymptotic behavior is essentially given by an interpolation between \eqref{NLS_2} and \eqref{Coulomb_2}.  

As an application of the above result, we also study the final state problem for the inhomogeneous NLS with an inverse-square potential: 
\begin{align}
\label{NLSI}
i\partial_t\tilde u+\Delta \tilde u+\frac{(n-2)^2}{4|x|^2}\tilde u=\tilde N(\tilde u),\quad(t,x)\in \R\times \R^n,
\end{align}
where $n\ge3$, $\lambda\in \R$ and $\tilde N(\tilde u)=\tilde \lambda|x|^{-\tilde \beta} |\tilde u|^{\tilde \alpha} \tilde u$. Precisely, under the radial symmetry, we show the modified scattering for a range of parameters $n,\tilde\alpha,\tilde\beta$ (see \eqref{long_range_NLSI} below). It is worth noting that the range includes the case $(n,\tilde \alpha,\tilde \beta)=(3,2/3,0)$, where \eqref{NLSI} becomes the standard long-range NLS with an inverse-square potential in the three space dimensions (see \eqref{NLSI0}). Since a seminal work by Burq et al \cite{BPST} on the Strichartz estimates, the well-posedness and short-range scattering theory for the standard NLS with inverse-square potentials $V(x)=a|x|^{-2}$ have been extensively studied (see e.g. \cite{BPST,ZhZh,KMVZZ,LMM,Suzuki}). However, there is no previous result on the modified scattering for such a model. Note that there are several results on the modified scattering for the NLS with other types of linear potentials (see e.g. \cite{Ivan,GPR,ChPu} for bounded decaying potentials in one space dimension, \cite{ShTo} for the Stark potential in dimension $1\le d\le3$ and \cite{Segata,MMS} for the Dirac delta potential in one space dimension). One interest of our result for \eqref{NLSI} compered with these previous literature is that we deal with a critically singular potential $|x|^{-2}$ in multi-dimensions. Also it is worth noting that the coupling constant $\frac{(n-2)^2}{4}$ is critical in the sense that the operator $\mathcal L=-\Delta-\frac{(n-2)^2}{4|x|^2}$ has a zero resonant state ({i.e.} non-$L^2$ ground state) $u_*=|x|^{-(n-2)/2}\notin L^2$, namely $u_*$ solves $\mathcal L u_*=0$ for $x\neq0$. The case with zero resonances (often called the exceptional case) is usually much harder to study than the generic case without zero resonances and there seems to be no previous result on the modified scattering for the multi-dimensional NLS in the exceptional case (except for the free case without external potential in $d=2$), while there are some known results in $d=1$ (see \cite{Ivan,ChPu} and references therein).

At the end of the introduction, we summarize the main results in  this paper:
\begin{itemize}
\item On the equation \eqref{INLS}: the modified scattering for the final state problem under the condition \eqref{long_range_INLS} (Theorem \ref{theorem_1}); the short-range scattering for the Cauchy problem in the case when $1<\alpha d/2+\beta\le2$ and $\lambda>0$ (Theorem \ref{theorem_C_1}).
\item On the equation \eqref{NLSI}: the modified scattering for the final state problem under the radial symmetry and the condition \eqref{long_range_NLSI} (Theorem \ref{theorem_2}). 
\end{itemize}
As mentioned above, these are the first results on the modified scattering for \eqref{INLS} and \eqref{NLSI}. We hope that our method could be also applied to establish the modified scattering for the Cauchy problems associated with \eqref{INLS} and \eqref{NLSI} in which case there is no existing result. 
 
\subsection{Notation}
\begin{itemize}
\item $\<x\>$ stands for $\sqrt{1+|x|^2}$. 
\item For $p\in [1,\infty]$, $p'=p/(p-1)$ denotes its H\"older conjugate exponent. 
\item For $a,b>0$, $a\lesssim b$ means there exists a (non-essential) constant $C>0$ such that $a\le Cb$.
\item $H^s=H^s(\R^d)$ and $\dot H^s=\dot H^s(\R^d)$ denote the inhomogeneous and homogeneous $L^2$-based Sobolev spaces of order $s$, respectively. 
\item $L^{p,q}=L^{p,q}(\R^d)$ denotes the Lorentz space (see \cite{Grafakos_1}). 
\item Given a function space $X$ on $\R^\ell$, $X_r$ denotes the subset of all radial functions in $X$. 
\end{itemize}
\subsection{Main results}
For the first result, we impose the following condition: 

%assumption
\begin{assumption}
\label{assumption_A}
$|x|^{-\beta/\alpha}\varphi$ belongs to $L^\infty$. Moreover, $\varphi$ satisfies the following properties: 
\begin{align}
\begin{cases}
\label{theorem_1_1}
|x|^{-\beta/\alpha}\varphi\in \dot H^\delta& \text{if $d=1$ and $\delta<1$},\\
|x|^{-1}\varphi\in H^{\delta-1}& \text{if $d=1$ and $\delta=1$ or $d=2,3$}
\end{cases}
\end{align}
\end{assumption}
Assumption \ref{assumption_A} holds if $\varphi\in H^\delta$ with $\delta>d/2$ and $\varphi$ is supported away from the origin. Therefore, Assumption \ref{assumption_A} is essentially a condition on the vanishing order of $\varphi$ at $x=0$ (see Remark \ref{remark_2} for more details). The result on the inhomogeneous NLS then is as follows: 

%theorem
\begin{theorem}	
\label{theorem_1}
Let $d\in \{1,2,3\}$, $\lambda\in \R$, $\alpha$ and $\beta$ satisfy \eqref{long_range_INLS}. Suppose $$\frac d2+\frac\beta\alpha<2\theta<\delta<\alpha+1$$ and $\delta\le d$. Then there exists $\ep>0$ such that for any $\varphi\in H^\delta$ satisfying Assumption \ref{assumption_A} and $\norm{|x|^{-\beta/\alpha}\varphi}_{L^\infty}\le\ep$ there exists a unique global solution $u\in C(\R;L^2(\R^d))$ to \eqref{INLS} satisfying
\begin{align}
\label{theorem_1_2}
\bignorm{u(t)-u_{\mathrm{ap}}(t)}_{L^2(\R^d)}\lesssim t^{-\theta},\quad t\ge1,
\end{align}
where the asymptotic profile $u_{\mathrm{ap}}$ is given by
\begin{align}
\label{theorem_1_3}
u_{\mathrm{ap}}(t,x)=\frac{1}{(2it)^{d/2}}\varphi\left(\frac{x}{2t}\right)\exp\left(i\frac{|x|^2}{4t}-i\frac\lambda2\left|\frac{x}{2t}\right|^{-\beta}\left|\varphi\left(\frac{x}{2t}\right)\right|^\alpha\log t\right).
\end{align}
\end{theorem}

Note that an analogous result also holds for the negative time direction. This remark also applies to Theorem \ref{theorem_2} below. We further make a few comments on this theorem.

%remark
\begin{remark}[Modified wave operator]
\label{remark_wave_operator}
Let $u_+=\F^{-1}\varphi$ and $S(t,x)=(\lambda/2) |x|^{-\beta}|\F u_+(x)|^\alpha \log t$. Then $u_{\mathrm{ap}}$ can be rewritten in the form
\begin{align*}
%\label{remark_2_ap}
u_{\mathrm{ap}}(t,x)
%=\mathcal M(t)\mathcal D(t)\F \F^{-1}e^{-iS(t,\cdot)}\F u_+
=\mathcal M(t)\mathcal D(t)\F e^{-iS(t,-i\nabla)} u_+,
\end{align*}
where $\M(t)$ and $\D(t)$ are the usual modulation and dilation operators (see \eqref{dilation} below).  Combining with the factorization formula $e^{it\Delta}=\M(t)\D(t)\F\M(t)$ (see \eqref{MDFM}), we have
$$
u_{\ap}=e^{it\Delta}e^{-iS(t,-i\nabla)}u_++o(1)
$$
 as $t\to\infty$ in $L^2$. Theorem \ref{theorem_1_1} thus shows that, given a $u_+\in X_\ep$, there exists a global solution $u(t)$ such that $e^{iS(t,-i\nabla)}e^{-it\Delta}u(t)$ converges to $u_+$ as $t\to \infty$ in $L^2$ provided $\ep>0$ is sufficiently small, where $X_\ep$ is defined by
$$
X_\ep:=\F^{-1}\{\varphi\in H^\delta\ |\ \text{$\varphi$ satisfies Assumption A and $\norm{|x|^{-\beta/\alpha}\varphi}_{L^\infty}\le \ep$}\}.
$$
 In particular, the modified wave operator $W^+_{\mathrm{\mathop{INLS}}}:X_\ep\ni u_+\mapsto u(0)\in L^2$ exists. 
\end{remark}

%remark
\begin{remark}
\label{remark_2}
%\item From the proof of Theorem \ref{theorem_1}, one sees that $u(t)$ also satisfies$$\norm{u-u_{\ap}}_{L^p((t,\infty);L^q(\R^d))}\lesssim t^{-\theta},\quad t\to\infty,$$where $(p,q)$ is equal to $(4,\infty)$ if $d=1$, $(3,6)$ if $d=2$ and $(2,6)$ if $d=3$, respectively. 
If $\varphi\in H^\delta$ is supported away from the origin then Assumption \ref{assumption_A} holds. Precisely, for any $\chi\in C_0^\infty(\R^d)$ satisfying $\chi\equiv1$ near the origin and any $\psi\in H^\delta$ with $\delta>d/2$,  $\varphi=(1-\chi)\psi$ satisfies Assumption \ref{assumption_A} as follows: Sobolev's embedding implies $|x|^{-\beta/\alpha}\varphi\in L^\infty$ and $|x|^{-1}\varphi\in L^2$. Moreover, for any $\sigma,\gamma\ge0$, $\<D\>^\sigma |x|^{-\gamma}(1-\chi(x))\<D\>^{-\sigma}$ is a pseudodifferential operator with a $S_{1,0}^0$-symbol and hence is bounded on $L^2$ by the Calder\'on--Vaillancourt theorem (see \cite[Chapter 8]{Grafakos_2}). With $(\sigma,\gamma)=(\beta/\alpha,\delta)$ if $d=1$ and $(\sigma,\gamma)=(1,\delta-1)$ if $d=2,3$, we thus have
$$
\norm{|x|^{-\gamma}\varphi}_{\dot H^{\sigma}}\lesssim \norm{\varphi}_{H^{\sigma}}
$$
and Assumption \ref{assumption_A} follows. It is also worth noting that, to ensure $|x|^{-\beta/\alpha}\varphi\in \dot H^\delta$ for general $\varphi\in H^\delta$ in $d=1$,  it is enough to assume that $|x|^{-\beta/\alpha-1}\varphi$ and $|x|^{-\beta/\alpha}\nabla \varphi$ belong to $L^2$ . Similarly, if $\varphi\in H^\delta$, $|x|^{-2}\varphi\in L^2$ and $|x|^{-1}\nabla\varphi \in L^2$ then $|x|^{-1}\varphi\in H^{\delta-1}$ in $d=2,3$. 
\end{remark}

As a supplementary result, we next state a result on the short-range scattering, which shows the condition  $\alpha d/2+\beta=1$ is critical from the view point of the scattering. Let $$\Sigma=\{u\in L^2(\R^d)\ |\ \norm{u}_{L^2}+\norm{\nabla u}_{L^2}+\norm{xu}_{L^2}<\infty\}.$$

%proposition
\begin{theorem}
\label{theorem_C_1}
Let $d\ge1$ and $\lambda>0$. Suppose that 
\begin{align}
\label{theorem_C_1_1}
0<\alpha\le \frac{4}{d},\quad 0\le \beta< \min(2,d),\quad 1<\frac{\alpha d}{2}+\beta\le 2.
\end{align}
Then, for any $u_0\in \Sigma$, there exists a unique global solution $u\in C(\R;\Sigma)$ to \eqref{INLS} with the initial condition $u(0)=u_0$ such that $u$ scatters to a free solution: there exist $u_{\pm}\in L^2$ such that
\begin{align}
\label{theorem_C_1_2}
\lim_{t\to\pm\infty}\norm{u(t)-e^{it\Delta}u_\pm}_{L^2}=0.
\end{align}
\end{theorem}

We next consider the equation \eqref{NLSI}. To state the result, we introduce a few notation. Let $$\L=-\Delta-\frac{(n-2)^2}{4|x|^2}$$ be the Schr\"odinger operator associated with \eqref{NLSI} and $\H^s=\<\L\>^{-s/2}L^2(\R^n)$ the Sobolev space adapted to $\L$ defined by the norm
$
\norm{f}_{\H^s}=\norm{\<\mathcal L\>^{s/2}f}_{L^2}
$. 
A remarkable fact is that $-\L$ is unitarily equivalent to the two-dimensional Laplacian on the space of radial functions, namely
\begin{align}
\label{equivalence}
\L f=-\Lambda^*\Delta_{\R^2}\Lambda f
\end{align}
for any radial $f\in \mathcal H^2$, where $\Lambda$ is a unitary from $L^2_r(\R^n)$ to $L^2_r(\R^2)$ given in \eqref{Lambda} below. We will give in Section \ref{section_NLSI} the precise definition and several basic properties of $\L$, as well as the characterization of $\H^s_r$, the space of all radial functions in $\H^s$. 

Using this equivalence, the analysis for \eqref{NLSI} under the radial symmetry can be reduced to that for \eqref{INLS}. More precisely, \eqref{NLSI} for radial functions is transformed into \eqref{INLS} with 
$$
d=2,\quad\alpha=\tilde\alpha,\quad \beta=\tilde\beta+\alpha(n-2)/2,\quad \lambda=c_n^{-\tilde\alpha}\tilde\lambda,
$$
where $c_n:={(2\pi)}^{-1/2}|\mathbb S^{n-1}|^{1/2}$. In particular, as an application of Theorem \ref{theorem_1}, we have the following result on the modified scattering. As in case of \eqref{INLS}, we impose the following condition: 

%assumption
\begin{assumption}
\label{assumption_B}
$|x|^{-\tilde \beta/\tilde \alpha}\tilde \varphi\in L^\infty$ and $|x|^{-1}\tilde \varphi\in \H^{\delta-1}_r$. 
\end{assumption}
As with Theorem \ref{theorem_1}, Assumption \ref{assumption_B} holds if $\tilde \varphi\in \H_r^\delta$ is supported away from the origin. 

%theorem
\begin{theorem}	
\label{theorem_2}
Suppose that $n\ge3$, $\tilde \lambda\in \R$ and $(\tilde \alpha,\tilde \beta)$ satisfies
\begin{align}
\label{long_range_NLSI}
-\frac{\tilde \alpha(n-2)}{2}\le \tilde \beta<1-\frac{\tilde \alpha(n-2)}{2},\quad \frac{\tilde \alpha n}{2}+\tilde \beta=1,\quad\frac n2+\frac{\tilde \beta}{\tilde \alpha}<\tilde \alpha+1.
\end{align}
Let $n/2+\tilde \beta/\tilde \alpha<2\theta<\delta<\tilde \alpha+1$. Then there exists $\ep>0$ such that for any radial $\tilde \varphi\in \H^\delta_r$ satisfying Assumption \ref{assumption_B} and $\norm{|x|^{-\tilde \beta/\tilde \alpha}\tilde \varphi}_{L^\infty}\le \ep$,  there exists a unique radial solution $\tilde u\in C(\R;L^2_r(\R^n))$ to \eqref{NLSI} satisfying the following asymptotic condition: 
\begin{align}
\label{theorem_2_1}
\norm{\tilde u(t)-\tilde u_{\ap}(t)}_{L^2(\R^n)}\lesssim t^{-\theta},\quad t\ge1,
\end{align}
where the asymptotic profile $\tilde u_{\ap}$ is defined by
\begin{align}
\label{theorem_2_2}
\tilde u_{\ap}(t,x)=\frac{1}{(2it)^{n/2}}\tilde \varphi\left(\frac{x}{2t}\right)\exp\left(i\frac{|x|^2}{4t}-i\frac{\tilde\lambda}{2}\left|\frac{x}{2t}\right|^{-\tilde \beta}\left|\tilde \varphi\left(\frac{x}{2t}\right)\right|^{\tilde \alpha}\log t\right).
\end{align}
\end{theorem}

We make a few comments on Theorem \ref{theorem_2}. 
%remark
\begin{remark}
\label{remark_1_6}
The assumption \eqref{long_range_NLSI} holds if and only if the following conditions are satisfied: 
$$
\frac{\sqrt{5}-1}{2}<\tilde \alpha\le 1,\quad -\frac{n-2}{2}\le \tilde \beta<\frac{n+4-n\sqrt5}{4},\quad \frac{\tilde \alpha n}{2}+\tilde \beta=1.
$$
The following two examples are of particular interest:
\begin{itemize}
\item Let $(\tilde \alpha,\tilde \beta,n)=(2/3,0,3)$. Then $\tilde N(\tilde u)$ becomes the standard long-range nonlinearity: $$\tilde N(\tilde u)=\tilde \lambda |\tilde u|^{2/3}\tilde u.$$ Hence \eqref{long_range_NLSI} covers the NLS with the critical  inverse-square potential in $n=3$. This is the first result on the modified scattering for the NLS with inverse-square potentials. 

\item Let $n\ge3$ and $(\tilde \alpha,\tilde \beta)=(1,-(n-2)/2)$. Then $\tilde N(\tilde u)$ becomes  a H\'enon-type nonlinearity: $$\tilde N(\tilde u)=\tilde \lambda |x|^{\frac{n-2}{2}}|\tilde u|\tilde u.$$
Such a type of nonlinearity has been extensively studied in the context of nonlinear elliptic and parabolic equations. However, it seems to be less understood in case of nonlinear Schr\"odinger equations. Moreover, it is also worth noting that we can take $n$ arbitrary large, while the space dimension $n$ is usually restricted to less than or equal to $3$ in case of the modified scattering for NLS with the power-type nonlinearity $|u|^{2/n}u$. 
\end{itemize}
\end{remark}

%remark
\begin{remark}[Modified wave operator]
As with Theorem \ref{theorem_1}, setting $\tilde u_+=\F^{-1}\tilde\varphi$, Theorem \ref{theorem_2} implies the existence of the modified wave operator $\tilde W^+_{\mathrm{\mathop{INLS}}}:\tilde u_+\mapsto \tilde u(0)$. 
\end{remark}

%remark
\begin{remark}
The potential $-\frac{(n-2)^2}{4|x|^{2}}$ is of short-range type in the sense that the wave operator $$W^+f:=\slim_{t\to\infty}e^{it\L}e^{it\Delta}f,\quad f\in L^2_r(\R^n)$$ exists and becomes a unitary on $L^2_r(\R^n)$. Indeed, the formulas \eqref{equivalence} and \eqref{MDFM} imply \begin{align*} W^+f&= \slim_{t\to\infty}\Lambda^*\M(t)^*\F_{\R^2}^*\D_{\R^2}(t)^*\M(t)^*\Lambda \M(t) \D_{\R^n}(t)\F_{\R^n}\M(t) f\\&=\slim_{t\to\infty}\Lambda^*\F_{\R^2}^*\D_{\R^2}(t)^*\Lambda \D_{\R^n}(t)\F_{\R^n}f\\&=i^{-(n-2)/2}\Lambda^*\F_{\R^2}^*\Lambda \F_{\R^n}f\end{align*} for $f\in L^2_r(\R^n)$, where $\D_{\R^\ell}(t)$ and $\F_{\R^\ell}$ are the $\ell$-dimensional dilation and Fourier transform (see Section \ref{section_2}), and we have used the fact $\M(t)\to 1$ strongly on $L^2$ as $t\to \infty$.  Hence the global dynamics of $e^{-it\L}$ is asymptotically governed by the free evolution $e^{it\Delta}$ at least under the radial symmetry; in fact, one can prove the existence and unitarity of $W^+$  without radial symmetry by using a different method (see \cite{Mizutani}). This is the reason why the asymptotic profile \eqref{theorem_2_2} is essentially the same as \eqref{theorem_1_2} and no influence of the linear potential $-\frac{(n-2)^2}{4|x|^{2}}$ appears.
\end{remark}

\subsection{Strategy of the proof}
Here we briefly outline the proof of Theorem \ref{theorem_1}, explaining especially where the assumptions come from. The proof basically follows a similar line as that of \cite{HaNa2006,HWN}. By a standard argument based on the global well-posedness in $L^2$ proved by \cite{Guzman}, the proof is reduced to construct a unique solution to an integral equation for sufficiently large $t\ge1$ by the contraction mapping theorem. The integral equation is roughly of the form 
$$u(t)=u_{\mathrm{ap}}(t)+i\int_t^\infty e^{i(t-s)\Delta}\left(N(u(s))-N(u_{\mathrm{ap}}(s))\right)ds+\text{(Remainder term)}.$$
To deal with the Duhamel integral term, we use the Strichartz estimates and the decay estimate $$\norm{|x|^{-\beta}|u_{\mathrm{ap}}(t)|^\alpha}_{L^\infty}\lesssim t^{-1}\norm{|x|^{-\beta} |\varphi|^\alpha}_{L^\infty}.$$To ensure the integrability in $s$ of an appropriate $L^q_x$-norm of $N(v(s))-N(u_{\mathrm{ap}}(s))$ appeared in the dual Strichartz norms, we need to work with a complete metric space in which it holds that $$\norm{u(t)-u_{\mathrm{ap}}(t)}_{L^2}\lesssim t^{-\theta}$$ with some $2\theta>d/2+\beta/\alpha$. Then the remainder terms should be small enough and $O_{L^2}(t^{-\theta})$ as $t\to \infty$. To ensure these properties, we show that the remainder terms are $O_{L^2}(t^{-\delta/2})$ with some $\delta>2\theta$. Let us explain in more details for one of the remainder terms, which is of the form $$\mathcal M(t)\mathcal D(t)\F(\mathcal M(t)-1)\F^{-1}N(\varphi(x))e^{i\mu |x|^{-\beta}|\varphi(x)|^\alpha},\quad \mu:=(\lambda/2) \log t,\ \mathcal M(t):=e^{\frac{i|x|^2}{4t}}.$$Using the bound $|\mathcal M(t)-1|\lesssim t^{-\delta/2}|x|^\delta$, the proof is then reduced to show the nonlinear estimate
\begin{align}
\label{outline_1}
\norm{N(\varphi(x))e^{\mu |x|^{-\beta}|\varphi(x)|^\alpha}}_{H^\delta}\le C_\varphi \<\mu\>^2
\end{align}
where, since $N(\varphi)=\lambda |x|^{-\beta}|\varphi|^\alpha\varphi$, we need to assume $\delta<\alpha+1$ to ensure that the left hand side of \eqref{outline_1} makes sense.  These restrictions on $\theta$ and $\delta$ lead to the condition $$d/2+\beta/\alpha<2\theta<\delta<\alpha+1.$$ To prove \eqref{outline_1}, we basically follow the same argument as in \cite[Section 2]{MMU} which is based on the fractional Leibniz rule and a Leibniz rule-type estimates for H\"older continuous functions by \cite{Visan}. One main new feature in our case is the variable coefficient term $|x|^{-\beta}$. To deal with $|x|^{-\beta}$ and its derivatives, we use fractional Hardy's inequality and the weighted $L^2$-boundedness for the Riesz transform. At this step, we need the first condition in \eqref{long_range_INLS} and the assumption \eqref{theorem_1_1}. 

\subsection{Organization of the paper}
The rest of the paper is devoted to proving the theorems stated above and organized as follows. In Section \ref{section_2}, we collect some basic tools used in the proofs, which include the factorization formula and Strichartz estimates for $e^{it\Delta}$. We also state a key lemma (Lemma \ref{lemma_w}) on the nonlinear estimates in Section \ref{section_2}. Section \ref{section_3} is devoted to the proof of Theorem \ref{theorem_1}. The proof of Theorem \ref{theorem_2} is given by Section \ref{section_NLSI}. In Section \ref{section_w}, we give the proof of the key lemma stated in Section \ref{section_2}. The proof of Theorem \ref{theorem_C_1} is given by Section \ref{section_C}. Appendix \ref{appendix_A} is concerned with a derivation of the integral equation  associated with the equation \eqref{INLS}. Finally, for the sake of self-containedness, we give the proof of the global existence of the $L^2$-solution to the Cauchy problem for \eqref{INLS} in Appendix \ref{appendix_B}. 

\section*{Acknowledgments}
T. Inui is partially  supported by JSPS KAKENHI Grant Number JP18K13444. H. Miyazaki is partially supported by JSPS KAKENHI Grant Number 19K14580. H. Mizutani is partially supported by JSPS KAKENHI Grant Number JP17K14218 and JP17H02854. K. Uriya is partially supported by JSPS KAKENHI Grant Number 19K14578. 

%section
\section{Preliminary materials}
\label{section_2}
In this section we collect several basic results used in the proof of Theorem \ref{theorem_1}. 

The free evolution group $e^{it\Delta}$ satisfies the following well-known factorization formula: 
\begin{align}
\label{MDFM}
e^{it\Delta}=\mathcal M(t)\mathcal D(t)\mathcal F\mathcal M(t),
\end{align}
where $\mathcal F$ denotes the Fourier transform on $\R^d$ and $\mathcal M=\mathcal M(t)$ and $\mathcal D=\mathcal D(t)$ are the modulation and dilation operators defined by
\begin{align}
\label{dilation}
\mathcal M(t)\varphi(x):=\exp\left(i\frac{|x|^2}{4t}\right)\varphi(x),\quad
\mathcal D(t)\varphi(x):=\frac{1}{(2it)^{d/2}}\varphi\left(\frac{x}{2t}\right).
\end{align}
Note that all of $e^{it\Delta},\mathcal F,\mathcal M$ and $\mathcal D$ are unitary maps on $L^2(\R^d)$. 
Given a function $\varphi$, we set
\begin{align}
\label{w}
w(t,x):=\varphi(x) \exp\left(-\frac{i}{2}\lambda |x|^{-\beta}|\varphi(x)|^\alpha\log t\right).
\end{align}
Using these notations, $u_{\mathrm{ap}}$ can be rewritten as $u_{\mathrm{ap}}(t)=[\mathcal M\mathcal D w](t)$. 

%{lemma}
\begin{lemma}
\label{lemma_M}
For any $1\le p\le\infty$, 
$$
\norm{\mathcal D(t)f}_{L^p}=|2t|^{-d(1/2-1/p)}\norm{f}_{L^p}.
$$
Moreover, for any $0\le \delta\le2$ and $1\le p\le\infty$, 
\begin{align*}
\norm{(\mathcal M(t)-1)f}_{L^p}\lesssim |t|^{-\delta/2}\norm{|x|^{\delta}f}_{L^p}.
\end{align*}
\end{lemma}

%proof
\begin{proof}
The first estimate is verified by a direct calculation, while the second estimate is deduced from the inequality 
$
|\mathcal M(t)-1|\le 2|\sin(|x|^2/{(8t)})|\lesssim |x|^{\delta}|t|^{-\delta/2}.
$
\end{proof}

Since $|x|^{-\beta}$ belongs to $L^{d/\beta,\infty}$ but $|x|^{-\beta}\notin L^p$ for any $p$, it is often convenient to work with the Lorentz spaces $L^{p,q}$ instead of the standard Lebesgue spaces $L^p$. We will use the following basic properties for $L^{p,q}$ (see \cite[Chapter 1]{Grafakos_1}):

%{lemma}
\begin{lemma}
\label{lemma_Holder}
The following properties hold: 
\begin{itemize}
\item Continuous embedding: For $1\le p<\infty$ and $1<r_1<p<r_2<\infty$, we have
$$
L^{p,1}\subset L^{p,r_1}\subset L^p=L^{p,p}\subset L^{p,r_2}\subset L^{p,\infty}.
$$
\item Homogeneity: $\norm{|f|^\alpha}_{L^{p,r}}=\norm{f}_{L^{p\alpha,r\alpha}}^\alpha$ for $1\le p<\infty$, $1\le r\le \infty$ and $\alpha>0$. 
\item H\"older's inequality: If $1\le p,p_1,p_2<\infty$, $1\le r,r_1,r_2\le \infty$, $1/p_1+1/p_2=1/p$ and $1/r_1+1/r_2=1/r$, then
\begin{align}
\label{Holder}
\norm{fg}_{L^{p,r}}\lesssim \norm{f}_{L^{p_1,r_1}}\norm{g}_{L^{p_2,r_2}},\quad
\norm{fg}_{L^{p,r}}\lesssim \norm{f}_{L^{\infty}}\norm{g}_{L^{p,r}}.
\end{align}

\end{itemize}
\end{lemma}

We also recall the Strichartz estimates. A pair $(p,q)$ is said to be admissible if 
$$
2\le p,q\le \infty,\quad \frac2p=d\left(\frac12-\frac1q\right),\quad (d,p,q)\neq (2,2,\infty).
$$

%{lemma}
\begin{lemma}
\label{lemma_Strichartz}
Let $(p,q)$ and $(\tilde p,\tilde q)$ be admissible pairs. For any interval $I$ containing $t=0$, \begin{align}\label{lemma_Strichartz_1}\norm{e^{it\Delta}f}_{L^{p}(\R;L^{q}(\R^d))}&\lesssim \norm{f}_{L^2(\R^d)},\\\label{lemma_Strichartz_2}\bignorm{\int_0^t e^{i(t-s)\Delta}F(s)ds}_{L^{p}(I;L^{q}(\R^d))}&\lesssim \norm{F}_{L^{\tilde p'}(I;L^{\tilde q'}(\R^d))},\end{align}
where the implicit constants are independent of $I$. For any $t>0$, we also have
\begin{align}
\label{lemma_Strichartz_3}
\bignorm{\int_\tau^\infty e^{i(\tau-s)\Delta}F(s)ds}_{L^{p}((t,\infty);L^{q}(\R^d))}&\lesssim \norm{F}_{L^{\tilde p'}((t,\infty);L^{\tilde q'}(\R^d))},
\end{align}
where the implicit constant is independent of $t$. Moreover, if in addition $q<\infty$ (resp. $\tilde q<\infty$), then $L^{q}$ (resp. $L^{\tilde q'}$) can be replaced by the Lorentz space $L^{q,2}$ (resp. $L^{\tilde q',2}$), respectively. 
\end{lemma}

%proof
\begin{proof}
All of \eqref{lemma_Strichartz_1}--\eqref{lemma_Strichartz_3} are well-known (see \cite{Strichartz,GiVe,Yajima,KeTa}). %For proving  \eqref{lemma_Strichartz_3}, we use \eqref{lemma_Strichartz_1} and its dual estimate to see that \eqref{lemma_Strichartz_2} holds with $I=\R$ and the integral over $[0,t]$ in the left hand side replaced by the integral over $[0,\infty]$. Combining with \eqref{lemma_Strichartz_2}, we then obtain $$\bignorm{\int_\tau^\infty e^{i(\tau-s)\Delta}F(s)ds}_{L^{p}(\R;L^{q}(\R^d))}\lesssim \norm{F}_{L^{\tilde p'}(\R;L^{\tilde q'}(\R^d))},$$which implies \eqref{lemma_Strichartz_3} since $\chi_{[t,\infty)}(\tau)=\chi_{[t,\infty)}(\tau)\chi_{[t,\infty)}(s)$ if $s\ge \tau\ge t$. 
An improvement to the analogous estimates involving with $L^{q,2}$ and $L^{\tilde q',2}$ instead of $L^{q}$ and $L^{\tilde q'}$ is due to \cite[Theorem 10.1]{KeTa}. 
\end{proof}

Next we record two basic estimates for the power-type nonlinearity $|z|^{\alpha}z$. 
%{lemma}
\begin{lemma}\label{lemma_nonlinear}Let $\alpha>0$. Then \begin{align}\label{lemma_nonlinear_1}\big||z_1|^\alpha z_1-|z_2|^\alpha z_2\big|\lesssim |z_2|^\alpha|z_1-z_2|+|z_1-z_2|^{\alpha+1},\quad z_1,z_2\in \C.\end{align}Moreover, for any $0<\alpha<1$ and any non-negative integer $n$, we also have \begin{align}\label{lemma_nonlinear_2}\big||z_1|^{\alpha-n} z_1^n-|z_2|^{\alpha-n} z_2^n\big|\lesssim |z_1-z_2|^{\alpha},\quad z_1,z_2\in \C.\end{align}\end{lemma}

%proof
\begin{proof}Although the lemma is well known (see e.g. \cite[Lemma 2.4]{GiVe_1989CMP}), we give the proof for the reader's convenience. To show \eqref{lemma_nonlinear_1}, we may assume  $|z_1|>|z_2|>0$ and compute \begin{align*}|z_1|^\alpha z_1-|z_2|^\alpha z_2=|z_1|^\alpha(z_1-z_2)+(|z_1|^\alpha-|z_2|^\alpha)z_2,\end{align*} where the desired estimate for the first term $|z_1|^\alpha(z_1-z_2)$ is easy to obtain. 
For the part $(|z_1|^\alpha-|z_2|^\alpha)z_2$ we consider two cases $\alpha\ge1$ and $\alpha<1$ separately. For $\alpha\ge1$,  we compute \begin{align*}|z_1|^\alpha-|z_2|^\alpha&=\alpha (|z_1|-|z_2|)\int_0^1\Big(|z_2|+\theta(|z_1|-|z_2|)\Big)^{\alpha-1}d\theta. \end{align*}
Since the integrand is dominated by $C(|z_2|^{\alpha-1}+|z_1-z_2|^{\alpha-1})$, we have %$$|z_2|^{\alpha-1}+\big||z_1|-|z_2|\big|^{\alpha-1}\le |z_2|^{\alpha-1}+|z_1-z_2|^{\alpha-1}.$$
%the term $(|z_1|^\alpha-|z_2|^\alpha)z_2$ satisfies
\begin{align*}\left|(|z_1|^\alpha -|z_2|^\alpha) z_2\right|
%&\lesssim |z_1-z_2|(|z_2|^{\alpha-1}+|z_1-z_2|^{\alpha-1})|z_2|\\
&\lesssim |z_1-z_2||z_2|^{\alpha}+|z_1-z_2|^\alpha |z_2|\Big(\chi_{\{|z_2|\le |z_1-z_2|\}}+\chi_{\{|z_2|\ge|z_1-z_2|\}}\Big)\\
&\lesssim |z_1-z_2|^{\alpha+1}+|z_1-z_2| |z_2|^{\alpha}\end{align*}
and \eqref{lemma_nonlinear_1} for $\alpha\ge1$ follows. If $0<\alpha<1$, then the concavity of the map $[0,\infty)\ni t\mapsto t^\alpha$ implies
\begin{align}
\label{lemma_nonlinear_proof_2}
\left||z_1|^\alpha-|z_2|^\alpha\right|\le  (|z_1|-|z_2|)^\alpha\le |z_1-z_2|^\alpha.
\end{align}
This proves \eqref{lemma_nonlinear_1} for $0<\alpha<1$. \eqref{lemma_nonlinear_proof_2} also implies  \eqref{lemma_nonlinear_2} in case of $n=0$. 

Next we shall prove \eqref{lemma_nonlinear_2} for $n\ge1$ by an induction argument. We  write
$$
|z_1|^{\alpha-1} z_1-|z_2|^{\alpha-1} z_2=\Big(|z_1|^{\alpha}-|z_2|^\alpha\Big)|z_1|^{-1}z_1+|z_2|^\alpha\Big(|z_1|^{-1}z_1-|z_2|^{-1}z_2\Big),
$$
where, by \eqref{lemma_nonlinear_proof_2}, the first term satisfies the desired estimate. For the second term, we compute
\begin{align*}
|z_2|^\alpha\Big(|z_1|^{-1}z_1-|z_2|^{-1}z_2\Big)
=|z_2|^{\alpha}|z_1|^{-1}|z_2|^{-1}\Big(|z_2|z_1-|z_1|z_2\Big).
\end{align*}
If $|z_2|\le |z_1-z_2|$, then the estimate \eqref{lemma_nonlinear_2} for $n=1$ follows. When $|z_2|\ge |z_1-z_2|$, we have
\begin{align}
\nonumber
|z_1|^{-1}|z_2|^{\alpha-1}\Big||z_2|z_1-|z_1|z_2\Big|
&=|z_1|^{-1}|z_2|^{\alpha-1}\Big||z_2|(z_1-z_2)+(|z_2|-|z_1|)z_2\Big|\\
\nonumber
&\le2|z_1|^{-1}|z_2|^{\alpha}|z_1-z_2|\\
\label{lemma_nonlinear_proof_3}
&\le 2|z_1-z_2|^\alpha
\end{align}
since $\alpha<1$ and $|z_1|>|z_2|$. This proves \eqref{lemma_nonlinear_2} for $n=1$. For $n\ge2$, if we write
\begin{align*}
&|z_1|^{\alpha-n} z_1^n-|z_2|^{\alpha-n} z_2^n\\
%&=\Big(|z_1|^{\alpha-n+1} z_1^{n-1}-|z_2|^{\alpha-n+1} z_2^{n-1}\Big)|z_1|^{-1}z_1+|z_2|^{\alpha-n+1} z_2^{n-1}|z_1|^{-1}z_1-|z_2|^{\alpha-n} z_2^n\\
%&=\Big(|z_1|^{\alpha-n+1} z_1^{n-1}-|z_2|^{\alpha-n+1} z_2^{n-1}\Big)|z_1|^{-1}z_1+|z_2|^{-n+1}z_2^{n-1}|z_2|^\alpha\Big(|z_1|^{-1}z_1-|z_2|^{-1}z_2\Big)\\
&=\Big(|z_1|^{\alpha-n+1} z_1^{n-1}-|z_2|^{\alpha-n+1} z_2^{n-1}\Big)|z_1|^{-1}z_1+|z_2|^{-n+1}z_2^{n-1}|z_1|^{-1}|z_2|^{\alpha-1}\Big(|z_2|z_1-|z_1|z_2\Big)
\end{align*}
then the first term in the right hand side is dominated by $C|z_1-z_2|^\alpha$ thanks to the hypothesis of induction, while the same argument as in the case $n=1$ shows that the second term of the right hand side is also dominated by $C|z_1-z_2|^\alpha$. This completes the proof of \eqref{lemma_nonlinear_2}. 
\end{proof}

Finally, we state a key lemma in  our argument. Let $E_{j,\delta}(\varphi)$ for $j=1,2$ be defined as follows:
\begin{itemize}
\item if $1/2<\delta<1$ then
\begin{align*}
E_1(\varphi)&:=\norm{\varphi}_{H^\delta},\\
E_2(\varphi)&:=\norm{|x|^{-\beta/\alpha}\varphi}_{L^\infty}^{\alpha}+\norm{|x|^{-\beta/\alpha}\varphi}_{L^\infty}^{\alpha-1}\norm{|x|^{-\beta/\alpha}\varphi}_{\dot H^\delta};
\end{align*}
\item if $\delta\ge1$ then
\begin{align*}
E_1(\varphi)&:=\norm{\varphi}_{H^{\delta}}+|\beta|\norm{|x|^{-1}\varphi}_{H^{\delta-1}},\\
E_2(\varphi)&:=\norm{|x|^{-\beta/\alpha}\varphi}_{L^\infty}^{\alpha}+\norm{|x|^{-\beta/\alpha}\varphi}_{L^\infty}^{\alpha-(\delta-1)/s}\Big(\norm{\varphi}_{H^{\delta}}+|\beta|\norm{|x|^{-1}\varphi}_{H^{\delta-1}}\Big)^{(\delta-1)/s},
\end{align*}
where $s=s(\delta,\alpha)$ is defined by the relation $2s=1+(\delta-1)/\alpha$. 
\end{itemize}
%\item For $\delta=1$, $$E_{1,1}(\varphi):=\norm{\varphi}_{H^{1}}+\norm{|x|^{-1}\varphi}_{L^2},\quad E_{2,1}(\varphi):=\norm{|x|^{-\beta/\alpha}\varphi}_{L^\infty}^\alpha.$$

%{lemma}
\begin{lemma}
\label{lemma_w}
Let $\alpha,\beta$ satisfy \eqref{long_range_INLS} and $\mu=(\lambda/2)\log t$. Then the following statements hold: 
\begin{itemize}
\item[(1)] Let $d=1$ and $1/2<\delta\le 1$. Then 
\begin{align*}
%\label{lemma_w_1}
\norm{\varphi e^{i\mu |x|^{-\beta}|\varphi|^\alpha}}_{H^\delta}&\lesssim E_1(\varphi)\Big(1+|\mu|E_2(\varphi)\Big),\\
\norm{N(\varphi)e^{i\mu |x|^{-\beta}|\varphi|^\alpha}}_{H^\delta}&\lesssim E_1(\varphi)E_2(\varphi)\Big(1+|\mu|E_2(\varphi)\Big). 
\end{align*}
\item[(2)] Let $d=2,3$ and $1<\delta<\alpha+1$. Then
\begin{align*}
\norm{\varphi e^{i\mu |x|^{-\beta}|\varphi|^\alpha}}_{H^{\delta}}&\lesssim E_1(\varphi)\Big(1+|\mu|E_2(\varphi)\Big)^2,\\
\norm{N(\varphi)e^{i\mu |x|^{-\beta}|\varphi|^\alpha}}_{H^{\delta}}&\lesssim E_1(\varphi)E_2(\varphi)\Big(1+\<\mu\>E_2(\varphi)\Big)^2.
\end{align*}
\end{itemize}
Here the implicit constants are independent of $\mu$ and $\varphi$. 
\end{lemma}
The proof of this lemma build on the argument by Masaki--Miyazaki--Uriya \cite[Section 2]{MMU}, which is based on the fractional Leibniz rule and a Leibniz rule-type estimate for H\"older continuous functions by Visan \cite[Appendix A]{Visan}. One new ingredient in our argument is to use the fractional Hardy inequality and the weighted $L^2$-boundedness of the Riesz transform to deal with the singular variable coefficient term $|x|^{-\beta}$ in the nonlinear term $N(\varphi)$. As the proof is rather involved, we postpone it to Section \ref{section_w}. 

%remark\begin{remark}If $\beta=0$, or equivalently $\alpha=2/d$, then Lemma \ref{lemma_w} implies the following: \begin{itemize}\item If $d=1$ and $1/2<\delta\le1$ then$$\norm{\varphi e^{i\mu|\varphi|^{2/d}}}_{H^\delta}+\norm{|\varphi|^{2/d}\varphi e^{i\mu|\varphi|^{2/d}}}_{H^\delta}\lesssim \<\mu\>\norm{\varphi}_{H^\delta}(1+\norm{\varphi}_{H^\delta})^{2\alpha}.$$\item If $d=2,3$ and $d/2<\delta<2/d+1$ then$$\norm{\varphi e^{i\mu|\varphi|^{2/d}}}_{H^\delta}+\norm{|\varphi|^{2/d}\varphi e^{i\mu|\varphi|^{2/d}}}_{H^\delta}\lesssim \<\mu\>^2\norm{\varphi}_{H^\delta}(1+\norm{\varphi}_{H^\delta})^{3\alpha}.$$\end{itemize}\end{remark}

%%%%%%%%%%
\section{Proof of Theorem \ref{theorem_1}}
\label{section_3}
This section is devoted to the proof of Theorem \ref{theorem_1}. Following the standard argument, we shall decompose the problem into two parts;  we first construct the solution $u\in C([T,\infty);L^2)$ to \eqref{INLS} satisfying \eqref{theorem_1_1} for sufficiently large $T>0$ and then we will extend it backward in time by using the well-posedness of the Cauchy problem for \eqref{INLS} established in Appendix \ref{appendix_B}. 

For the former  problem, we define a nonlinear map $\Psi[v]$ by
\begin{align}
\label{Phi}
\Psi[v](t):=u_{\mathrm{ap}}(t)+\mathcal K_1[v](t)+\mathcal K_2(t)+\mathcal Rw(t)
\end{align}
where $u_{\mathrm{ap}}$ and $w$ are given by \eqref{theorem_1_3} and \eqref{w}, respectively,  and 
\begin{align}
\nonumber
\mathcal K_1[v](t)&:=i\int_t^\infty e^{i(t-s)\Delta}\Big(N(v(s))-N(u_{\mathrm{ap}}(s))\Big)ds,\\
\nonumber
\mathcal K_2(t)&:=-i\int_t^\infty e^{i(t-s)\Delta}\mathcal R(s)N(w(s))\frac{ds}{2s},\\
\label{R}
\mathcal R=\mathcal R(t)&:=\mathcal M(t)\mathcal D(t)\mathcal F\left(\mathcal M(t)-1\right)\mathcal F^{-1}.
\end{align}
Then the equation \eqref{INLS} subjected to the asymptotic condition \eqref{theorem_1_2} can be reformulated as
\begin{align}
\label{integral_equation}
u=\Psi[u]
\end{align}
(see Appendix \ref{appendix_A} below). We shall solve this integral equation by means of the contraction mapping theorem. To this end we introduce a complete metric space
$$
\mathscr{X}(\theta,\rho,T):=\{v\in C([T,\infty);L^2(\R^d))\ |\ \norm{v-u_{\mathrm{ap}}}_{\mathscr X}\le \rho\}
$$
equipped with the distance function $d(u,v)=\norm{u-v}_{\mathscr X}$, where
\begin{align*}
\norm{v}_{\mathscr X}=\sup_{t\ge T}t^\theta \norm{f(t)}_{\mathscr Y(t)},\quad
\norm{f(t)}_{\mathscr Y(t)}=\norm{f(t)}_{L^2(\R^d)}+\norm{f}_{L^p((t,\infty);L^q(\R^d))}.
\end{align*}
Here we take the admissible pair $(p,q)$ as
\begin{align}
\label{admissible}
(p,q)=\begin{cases}(4,\infty)&\text{if}\ d=1,\\(3,6)&\text{if}\ d=2,\\(2,6)&\text{if}\ d=3.\end{cases}
\end{align}

We start dealing with the nonlinear part $\mathcal K_1[v]$. 

%{proposition}
\begin{proposition}
\label{proposition_3_1}
Suppose that $2\theta>d/2+{\beta}/{\alpha}$. Then, for any $t\ge T>0$ and $v,v_1,v_2\in X_\rho$, 
\begin{align}
\label{proposition_3_1_1}
\norm{\mathcal K_1[v]}_{\mathscr Y(t)}&\lesssim \rho t^{-\theta} \left(\rho^{\alpha} t^{-\alpha\theta+1/2}+\norm{|x|^{-\beta/\alpha}\varphi}_{L^\infty}^\alpha\right),\\
\label{proposition_3_1_2}
\norm{\mathcal K_1[v_1]-\mathcal K_1[v_2]}_{\mathscr Y(t)}&\lesssim \rho t^{-\theta} \left(\rho^{\alpha} t^{-\alpha\theta+1/2}+\norm{|x|^{-\beta/\alpha}\varphi}_{L^\infty}^\alpha\right).
\end{align}
\end{proposition}
 Note that $\alpha\theta>1/2$ since $1/\alpha=d/2+\beta/\alpha$ under the assumption \eqref{long_range_INLS}. 
%proof
\begin{proof}
Throughout the proof we set $I_t=(t,\infty)$ for short. Let
\begin{align*}
N_1(v,u_{\mathrm{ap}})&:=\chi_{\{|v-u_{\mathrm{ap}}|\le |u_{\mathrm{ap}}|\}}(x)\left(N(v)-N(u_{\mathrm{ap}})\right),\\
N_2(v,u_{\mathrm{ap}})&:=\chi_{\{|v-u_{\mathrm{ap}}|\ge |u_{\mathrm{ap}}|\}}(x)\left(N(v)-N(u_{\mathrm{ap}})\right).
\end{align*}
Then $N(v)-N(u_{\mathrm{ap}})=N_1(v,u_{\mathrm{ap}})+N_2(v,u_{\mathrm{ap}})$ and, by virtue of Lemma \ref{lemma_nonlinear}, $N_1$ and $N_2$ satisfy
\begin{align*}
|N_1(v,u_{\mathrm{ap}})|\lesssim |x|^{-\beta}|v-u_{\mathrm{ap}}||u_{\mathrm{ap}}|^\alpha,\quad
|N_2(v,u_{\mathrm{ap}})|\lesssim |x|^{-\beta}|v-u_{\mathrm{ap}}|^{\alpha+1}.
\end{align*}

For the part $N_1$, we use the Strichartz estimate \eqref{lemma_Strichartz_3} with $(\tilde p',\tilde q')=(1,2)$ to obtain
\begin{align*}
\bignorm{\int_\tau^\infty e^{i(\tau-s)\Delta}N_1(v,u_{\mathrm{ap}})(s)ds}_{\mathscr Y(t)}
&\lesssim \norm{|x|^{-\beta}|v-u_{\mathrm{ap}}||u_{\mathrm{ap}}|^\alpha}_{L^1(I_t;L^2)}\\
&\lesssim \bignorm{ \norm{v-u_{\mathrm{ap}}}_{L^2}\norm{|x|^{-\beta}|u_{\mathrm{ap}}|^\alpha}_{L^\infty}}_{L^1(I_t)},
\end{align*}
where $\norm{v(s)-u_{\mathrm{ap}}(s)}_{L^2}\le \rho s^{-\theta}$ by the hypothesis $v\in\mathscr{X}(\theta,\rho,T)$. Moreover, since
$$
|x|^{-\beta}|u_{\mathrm{ap}}(s,x)|^\alpha=\frac{1}{(2s)^{{\alpha d}/{2}+\beta}}\left|\frac{x}{2s}\right|^{-\beta} \left|\varphi \left(\frac{x}{2s}\right)\right|^\alpha=\frac{1}{2s}\left|\frac{x}{2s}\right|^{-\beta} \left|\varphi \left(\frac{x}{2s}\right)\right|^\alpha
$$
by the condition $\alpha d/2+\beta=1$, as long as $\theta>0$, $N_1$ satisfies
\begin{align}
\nonumber
\bignorm{\int_\tau^\infty e^{i(\tau-s)\Delta}N_1(v,u_{\mathrm{ap}})(s)ds}_{\mathscr Y(t)}
&\lesssim \rho\int_t^\infty s^{-\theta-1}\norm{|x|^{-\beta/\alpha}\varphi}_{L^\infty}^\alpha ds\\
\label{proposition_3_1_proof_1}
&\lesssim \rho t^{-\theta}\norm{|x|^{-\beta/\alpha}\varphi}_{L^\infty}^\alpha.
\end{align}

For the part $N_2$, we shall show the following estimate
\begin{align}
\label{proposition_3_1_proof_3}
\bignorm{\int_\tau^\infty e^{i(\tau-s)\Delta}N_2(v,u_{\mathrm{ap}})(s)ds}_{\mathscr Y(t)}\lesssim \rho t^{-\theta}\rho^{\alpha }t^{-\alpha \theta+1/2}
\end{align}
which, together with \eqref{proposition_3_1_proof_1}, implies \eqref{proposition_3_1_1}. We deal with the cases $d=1,2$ and $3$ separately. Note that, since $\beta+\alpha d/2=1$, the condition $2\theta>d/2+\beta/\alpha$ is equivalent to the inequality $2\alpha\theta>1$.

Let first $d=1$. Taking $\ep>0$ so small that $2(\alpha-2\ep)\theta >1$, we define $p_1,q_1,r_1$ by the relations:\begin{align*}
\frac{1}{p_1}=\frac14-\frac\ep2,\quad \frac{1}{q_1}=\ep,\quad \frac{1}{r_1}=\frac{\alpha }{2}-\ep. 
\end{align*}
Note that $(p_1,q_1)$ is admissible, $r_1<2$ for small $\ep$ since $\alpha>1$ and $1/r_1=1/q_1'-\beta$. Then the Strichartz estimate with $(p,q,\tilde p,\tilde q)=(4,\infty,p_1,q_1)$ (see Lemma \ref{lemma_Strichartz} above) implies
\begin{align*}
\bignorm{\int_\tau^\infty e^{i(\tau-s)\Delta}N_2(v,u_{\mathrm{ap}})(s)ds}_{\mathscr Y(t)}
&\lesssim\norm{|x|^{-\beta}|v-u_{\mathrm{ap}}|^{\alpha+1}}_{L^{p_1'}(I_t;L^{q_1',2})}\\
&\lesssim \norm{|x|^{-\beta}}_{L^{1/\beta,\infty}}\norm{|v-u_{\mathrm{ap}}|^{\alpha+1}}_{L^{p_1'}(I_t;L^{r_1,2})}\\
&\lesssim \norm{|v-u_{\mathrm{ap}}|^{\alpha+1}}_{L^{p_1'}(I_t;L^{r_1,2})}.
\end{align*}
Since $L^{r_1}\subset L^{r_1,2}$ (see Lemma \ref{lemma_Holder} above), we have
\begin{align*}
\norm{|v-u_{\mathrm{ap}}|^{\alpha+1}}_{L^{r_1,2}}
&\le \norm{|v-u_{\mathrm{ap}}|^{\alpha-2\ep}}_{L^{r_1}}\norm{|v-u_{\mathrm{ap}}|^{1+2\ep}}_{L^\infty}\\
&=\norm{v-u_{\mathrm{ap}}}_{L^2}^{\alpha-2\ep}\norm{v-u_{\mathrm{ap}}}_{L^\infty}^{1+2\ep}.
\end{align*}
This, together with the hypothesis $v\in \mathscr{X}(\theta,\rho,T)$, implies
\begin{align}
\nonumber
\norm{|v-u_{\mathrm{ap}}|^{\alpha+1}}_{L^{p_1'}(I_t;L^{r_1,2})}
%\nonumber&\lesssim \bignorm{\norm{v-u_{\mathrm{ap}}}_{L^2}^{\alpha-2\ep}\norm{v-u_{\mathrm{ap}}}_{L^\infty}^{1+2\ep}}_{L^{p_1'}(I_t)}\\
\nonumber
&\lesssim \bignorm{\norm{v-u_{\mathrm{ap}}}_{L^2}^{\alpha-2\ep}}_{L^2(I_t)}\bignorm{\norm{v-u_{\mathrm{ap}}}_{L^\infty}^{1+2\ep}}_{L^{\frac{4}{1+2\ep}}(I_t)}\\
\nonumber&\lesssim \norm{v-u_{\mathrm{ap}}}_{L^{2(\alpha-2\ep)}(I_t;L^2)}^{\alpha-2\ep}\norm{v-u_{\mathrm{ap}}}_{L^4(I_t;L^\infty)}^{1+2\ep}\\
\nonumber
&\lesssim \left(\int_t^\infty (\rho s^{-\theta})^{2(\alpha-2\ep)}ds\right)^{1/2}(\rho t^{-\theta})^{1+2\ep}\\
\nonumber
&\lesssim \rho^{\alpha+1}t^{-\alpha\theta-\theta+1/2}
\end{align}
and \eqref{proposition_3_1_proof_3} for $d=1$ follows, where we have used the relation $1/p_1'-(1+2\ep)/4=1/2$. 

We next let $d=2$. With the inequality $2(\alpha+1)\theta >1$ at hand, we take $0<\ep<1$ so small that 
$
{2(\alpha+1-3\ep)\theta }>1
$ and define the exponents $p_2,q_2,r_2$ and $r_3$ by the relations
\begin{align*}
\frac{1}{p_2}=\frac12-\ep,\quad \frac{1}{q_2}=\ep,\quad \frac{1}{r_2}=\frac{\alpha+1}{2}-\ep=\frac{1}{q_2'}-\frac \beta2,\quad \frac{1}{r_3}=\frac{1}{r_2}-\frac\ep2.
\end{align*}
Since $(p_2,q_2)$ is admissible and $q_2<\infty$, Lemma \ref{lemma_Strichartz} implies
\begin{align*}
\bignorm{\int_\tau^\infty e^{i(\tau-s)\Delta}N_2(v,u_{\mathrm{ap}})(s)ds}_{\mathscr Y(t)}
&\lesssim\norm{|x|^{-\beta}|v-u_{\mathrm{ap}}|^{\alpha+1}}_{L^{p_2'}(I_t;L^{q_2',2})}\\
&\lesssim \norm{|x|^{-\beta}}_{L^{2/\beta,\infty}}\norm{|v-u_{\mathrm{ap}}|^{\alpha+1}}_{L^{p_2'}(I_t;L^{r_2,2})}.
\end{align*}
Since $r_2<2$ and $\alpha+1-3\ep=2/r_3$, the same argument as that in case of $d=1$ implies
\begin{align}
\nonumber
\norm{|v-u_{\mathrm{ap}}|^{\alpha+1}}_{L^{p_2'}(I_t;L^{r_2,2})}
\nonumber
&\lesssim \bignorm{\norm{|v-u_{\mathrm{ap}}|^{\alpha+1-3\ep}}_{L^{r_3}}\norm{|v-u_{\mathrm{ap}}|^{3\ep}}_{L^{2/\ep}}}_{L^{p_2'}(I_t)}\\
%\nonumber&= \bignorm{\norm{v-u_{\mathrm{ap}}}_{L^2}^{\alpha+1-3\ep}\norm{v-u_{\mathrm{ap}}}_{L^6}^{3\ep}}_{L^{p_2'}(I_t)}\\
\nonumber
&\le \bignorm{\norm{v-u_{\mathrm{ap}}}_{L^2}^{\alpha+1-3\ep}}_{L^2(I_t)}\bignorm{\norm{v-u_{\mathrm{ap}}}_{L^6}^{3\ep}}_{L^{1/\ep}(I_t)}\\
\nonumber&=\norm{v-u_{\mathrm{ap}}}_{L^{2(\alpha+1-3\ep)}(I_t;L^2)}^{\alpha+1-3\ep}\norm{v-u_{\mathrm{ap}}}_{L^3(I_t;L^6)}^{3\ep}\\
\nonumber
&\lesssim \left(\int_t^\infty(\rho s^{-\theta})^{2(\alpha+1-3\ep)}ds\right)^{1/2}(\rho t^{-\theta})^{3\ep}\\
\nonumber
&\lesssim \rho^{\alpha+1}t^{-(\alpha+1)\theta+1/2}
\end{align}
and \eqref{proposition_3_1_proof_3} for $d=2$ follows, where we have used the relation $1/p_2'=1/2+\ep$ in the second line. 

When $d=3$, we use the double endpoint Strichartz estimate to obtain
\begin{align*}
\bignorm{\int_\tau^\infty e^{i(\tau-s)\Delta}N_2(v,u_{\mathrm{ap}})(s)ds}_{\mathscr Y(t)}
&\lesssim \norm{|x|^{-\beta}|v-u_{\mathrm{ap}}|^{\alpha+1}}_{L^2(I_t;L^{\frac{6}{5},2})}\\
&\lesssim \norm{|x|^{-\beta}}_{L^{\frac 3\beta,\infty}}\norm{|v-u_{\mathrm{ap}}|^{\alpha+1}}_{L^2(I_t;L^{{\frac{2}{\alpha+1}},2})},
\end{align*}
where we have used the fact $(d+2-2\beta)/d=\alpha+1$ (with $d=3$). Moreover, since $L^{\frac{2}{\alpha+1}}\subset L^{\frac{2}{\alpha+1},2}$ and $2(\alpha+1)\theta>1$, we can estimate the last term as follows:
\begin{align}
\nonumber
\norm{|v-u_{\mathrm{ap}}|^{\alpha+1}}_{L^2((t,\infty);L^{\frac{2}{\alpha+1},2})}
%&\le \norm{|v-u_{\mathrm{ap}}|^{\frac{d+2-2\beta}{d}}}_{L^2(I_t;L^{\frac{2d}{d+2-2\beta}})}\\\nonumber
&\lesssim\left(\int_t^\infty \norm{v(s)-u_{\mathrm{ap}}(s)}_{L^2}^{{2(\alpha+1)}}ds\right)^{1/2}\\
\nonumber
&\lesssim \left(\int_t^\infty (\rho s^{-\theta})^{2(\alpha+1)}ds\right)^{1/2}\\
\nonumber
&\lesssim \rho^{\alpha+1}t^{-(\alpha+1)\theta+1/2}.
\end{align}
This proves \eqref{proposition_3_1_proof_3} for $d=3$ and completes the proof of \eqref{proposition_3_1_1}. Finally, the same argument also implies \eqref{proposition_3_1_2} since
$
N(v_1)-N(v_2)
%&=N_1(v_1)-N(u_{\mathrm{ap}})+N(u_{\mathrm{ap}})-N(v_2)\\
=N_1(v_1,u_{\mathrm{ap}})+N_2(v_2,u_{\mathrm{ap}})-N_1(v_2,u_{\mathrm{ap}})-N_2(v_2,u_{\mathrm{ap}})$. 
\end{proof}

Next we summarize the decay estimates for the inhomogeneous parts $\mathcal Rw$ and $\mathcal K_2$. 

%{lemma}
\begin{proposition}
\label{proposition_3_2}
Let $\alpha,\beta$ and $\delta$ be as in Theorem \ref{theorem_1}. Then, for any $\delta'<\delta$,
\begin{align*}
\norm{\mathcal Rw(t)}_{\mathscr Y(t)}+\norm{\mathcal K_2(t)}_{\mathscr Y(t)}\lesssim t^{-\delta'/2}E_1(\varphi)\<E_2(\varphi)\>^3,\quad t\ge1. 
\end{align*}
\end{proposition}

%proof
\begin{proof}
We first use Lemma \ref{lemma_w} with $\mu=\lambda(\log t)/2$ to find
\begin{align}
\label{proposition_3_2_proof_1}
\norm{w(t)}_{H^\delta}+\norm{N(w(t))}_{H^\delta}\lesssim \<\log t\>^2E_1(\varphi)\<E_2(\varphi)\>^3.
\end{align}
Assume without loss of generality that $\ep:=(\delta-\delta')/2>0$ is sufficiently small such that $\delta'>d/2+\beta/\alpha$. Recall that $\mathcal R=\mathcal M\mathcal D \mathcal F(\mathcal M-1)\mathcal F^{-1}$. For any admissible pair $(p,q)$ and $\ep\le \gamma\le 2$, Lemma \ref{lemma_M} and the Hausdorff-Young inequality then imply 
\begin{align*}
\norm{\mathcal Rw(t)}_{L^q}
&\lesssim t^{-2/p}\norm{(\mathcal M-1)\mathcal F^{-1}w(t)}_{L^{q'}}\\
%&\lesssim t^{-2/p}\norm{\<x\>^{d(1/2-1/q)+\ep} (\mathcal M-1)\F^{-1}w(t)}_{L^{2}}\\
&\lesssim t^{-2/p-\gamma/2+\ep/2}\norm{|x|^{\gamma-\ep}\<x\>^{2/p+\ep}\F^{-1}w(t)}_{L^{2}}\\
&\lesssim t^{-2/p-\gamma/2+\ep/2}\norm{w(t)}_{H^{\gamma+2/p}},
\end{align*}
where we have used the relation $2/p=d(1/2-1/q)$. Letting $(\gamma,p,q)=(\delta,\infty,2)$, \eqref{proposition_3_2_proof_1} implies
\begin{align*}
\norm{\mathcal Rw(t)}_{L^2}
\lesssim t^{-\delta/2+\ep/2}\norm{w(t)}_{H^{\delta}}
%&\lesssim t^{-\delta'/2-\ep/2}\<\log t\>^2E_1(\varphi)\<E_2(\varphi)\>^2\\
\lesssim t^{-\delta'/2}E_1(\varphi)\<E_2(\varphi)\>^3
\end{align*}
since $t^{-\ep/2}\<\log t\>^2\lesssim1$. Choosing $p$ and $q$ as in \eqref{admissible} and $\gamma=\delta-2/p>0$, we similarly obtain
\begin{align*}
\norm{\mathcal Rw}_{L^p((t,\infty);L^q)}
&\lesssim \left(\int_t^\infty s^{-1-p\delta /2+\ep p/2}\norm{w(s)}_{H^{\delta}}^pds\right)^{1/p}\\
&\lesssim t^{-\delta/2+\ep/2}\norm{w(t)}_{H^{\delta}}\\
&\lesssim t^{-\delta'/2}E_1(\varphi)\<E_2(\varphi)\>^3.
\end{align*}
These two bounds for $\mathcal Rw(t)$ imply the desired estimate for the part $\norm{\mathcal Rw(t)}_{\mathscr Y(t)}$. 

Similarly, it follows from the Strichartz estimate and  \eqref{proposition_3_2_proof_1} that
\begin{align*}
\norm{\mathcal K_2(t)}_{\mathscr Y(t)}
&\lesssim \int_t^\infty\norm{\mathcal RN(w(s))}_{L^2}\frac{ds}{2s}\\
&\lesssim \int_t^\infty s^{-1-\delta/2+\ep/2}\norm{N(w(s))}_{H^\delta}ds\\
&\lesssim t^{-\delta'/2}E_1(\varphi)\<E_2(\varphi)\>^3.
\end{align*}
This completes the proof.
\end{proof}

We are now in a position to prove Theorem \ref{theorem_1}.

%proof
\begin{proof}[Proof of Theorem \ref{theorem_1}]
Let $\delta,\theta$ and $\varphi$ be as in Theorem \ref{theorem_1} in which case $E_1(\varphi)$ and $E_2(\varphi)$ are finite. Let $\delta'$ be such that $2\theta<\delta'<\delta$. Since $\alpha\theta>1/2$, Propositions \ref{proposition_3_1} and \ref{proposition_3_2} show that, for any $\rho>0$, there exist $\ep,T>0$ such that if $\norm{|x|^{-\beta/\alpha}\varphi}_{L^\infty}\le \ep$, then the map $v\mapsto  \Psi[v]$ is a contraction map in $\mathscr X(\theta,\rho,T)$ and hence \eqref{integral_equation} admits a unique solution $u\in\mathscr X(\theta,\rho,T)$. 

In order to extend $u$ backward in time, by virtue of the unique global existence for the Cauchy problem \eqref{INLS} with the initial condition $u_0=u(T)$  (see Proposition \ref{proposition_B_1} and Remark \ref{remark_appendix_B_2} below), it is sufficient to show that $u\in L^{p}_{\loc}([T,\infty);L^q)$ for any admissible pair $(p,q)$ and $u$ satisfies
\begin{align}
\label{theorem_1_proof_1}
u(t+T)=e^{it\Delta}u(T)-i\int_0^te^{i(t-s)\Delta}N(u(s+T))ds,\quad t\ge0.
\end{align}
We first show the former assertion. By Lemma \ref{lemma_Strichartz}, \eqref{proposition_3_1_proof_1} and \eqref{proposition_3_1_proof_3} still hold with $\mathscr Y(t)$ replaced by $L^{p}_{\loc}([T,\infty);L^q)$ and hence $\mathcal K_1(u)\in L^{p}_{\loc}([T,\infty);L^q)$. It is also easily seen from the proof of Proposition \ref{proposition_3_2} that $\mathcal R w,\mathcal K_2\in L^{p}_{\loc}([T,\infty);L^q)$. Furthermore, Lemma \ref{lemma_w} and Sobolev's embedding imply $u_{\ap} \in L^\infty_{\loc}([T,\infty);H^\delta)\subset L^{p}_{\loc}([T,\infty);L^q)$ since $\delta>d/2$. This proves $u\in L^{p}_{\loc}([T,\infty);L^q)$. Next, in order to verify \eqref{theorem_1_proof_1}, using \eqref{appendix_A_1}--\eqref{appendix_A_3} we rewrite \eqref{integral_equation} as
\begin{align}
\nonumber
u(t+T)
&= e^{i(t+T)\Delta}\F^{-1}w(t+T)\\
\label{theorem_1_proof_2}
&+i\int_t^\infty e^{i(t-s)\Delta}\left[N(u(s+T))-\frac{e^{i(s+T)\Delta}}{2(s+T)}\F^{-1}N(w(s+T))\right]ds. 
\end{align}
In particular, $u(T)$ is written in the form
\begin{align*}
%\label{theorem_1_proof_3}
u(T)=e^{iT\Delta}\F^{-1}w(T)+i\int_0^\infty e^{-is\Delta}\left[N(u(s+T))-\frac{e^{i(s+T)\Delta}}{2(s+T)}\F^{-1}N(w(s+T))\right]ds.
\end{align*}
Let $I$ be the difference of the right hand sides of \eqref{theorem_1_proof_1} and \eqref{theorem_1_proof_2}. Then we have for $t\ge0$,
\begin{align*}
I&=e^{it\Delta} u(T)-i\int_0^\infty e^{i(t-s)\Delta}\left[N(u(s+T))-\frac{e^{i(s+T)\Delta}}{2(s+T)}\F^{-1}N(w(s+T))\right]ds\\
&+i\int_0^t\frac{e^{i(t+T)\Delta}}{2(s+T)}\F^{-1}N(w(s+T))ds-e^{iT\Delta}\F^{-1}w(t+T)\\
&=e^{i(t+T)\Delta}\F^{-1}\left[w(T)-w(t+T)+i\int_T^{t+T}\frac{1}{2s}N(w(s))ds\right]=0
\end{align*}
since $w(t)$ solves $i\partial_t w(t)=(2t)^{-1}N(w(t))$. Hence \eqref{theorem_1_proof_2} implies \eqref{theorem_1_proof_1}. Now we know by Proposition \ref{proposition_B_1} and Remark \ref{remark_appendix_B_2} that $u\in C(\R;L^2)$. This completes the proof.
\end{proof}

%%%%%%%%%%%%%%%%
%section
\section{Proof of Theorem \ref{theorem_2}}
\label{section_NLSI}
This section is devoted to the proof of Theorem \ref{theorem_2}. We start with the precise definition of the operator $\L$, as well as several properties of its radial part. Let $n\ge3$ and set $D_*=C_0^\infty(\R^n\setminus0)$ for short. Thanks to the sharp Hardy inequality  of the form
\begin{align}
\label{Hardy}
\frac{(n-2)^2}{4}\int \frac{|f|^2}{|x|^2}dx\le \int |\nabla f|^2dx,\quad f\in D_*,
\end{align}
the sesquilinear form 
$$
Q_{\mathcal L}(f,g)=\int \Big(\nabla f\cdot\overline{\nabla g}-\frac{(n-2)^2}{4|x|^2}f\overline g\Big)dx,\quad f,g\in D_*,
$$
is closable and non-negative: $Q_{\mathcal L}(f,f)\ge0$. We denote by the same symbol $Q_{\mathcal L}$ its closed extension with domain $D(Q_\mathcal L)$, which is the completion of $D_*$ with respect to the norm $$\sqrt{Q_{\mathcal L}(f,f)+\norm{f}_{L^2}}.$$ Then we define $\L$ as the Friedrichs extension of $Q_{\mathcal L}$ with domain $D(\mathcal L)$ given by
$$
D(\mathcal L)=\{f\in D(Q_\mathcal L)\ |\ |Q_\mathcal L(f,g)|\le C_f\norm{g}_{L^2}\ \text{for all}\ g\in L^2\}
$$
(see \cite[Theorem X.23]{ReSi_II}). Note that $\mathcal L$ is a non-negative self-adjoint operator on $L^2$ such that $D_*\subset D(\mathcal L)\subset D(Q_\mathcal L)$ and that $\L$ is an extension of $-\Delta-\frac{(n-2)^2}{4|x|^2}$ defined on $D_*$. 

With the identification $L^2_r(\R^n)\cong L^2(\R_+,r^{n-1}dr)$, we denote by $f(r)$ the element in $L^2_r(\R^n)$ without any confusion. Define a multiplication operator $\Lambda$ by 
\begin{align}
\label{Lambda}
\Lambda f(r):=c_n r^{\frac{n-2}{2}} f(r),\quad c_n:={(2\pi)}^{-1/2}|\mathbb S^{n-1}|^{1/2},
\end{align}
for any radial function $f$, where $|\mathbb S^{n-1}|$ is the volume of the unit sphere $\mathbb S^{n-1}$. It is easy to see that $\Lambda$ is a unitary from $L^2_r(\R^n)$ to $L^2_r(\R^2)$ and its adjoint is given by $\Lambda^*=c_n^{-1}r^{-\frac{n-2}{2}}$. Note that $u_*(x)=|x|^{-\frac{n-2}{2}}$ is a virtual minimizer of \eqref{Hardy}, namely $u_*$ solves 
$$
-\Delta u_*(x)-\frac{(n-2)^2}{4|x|^2}u_*(x)=0,\quad x\neq0.
$$
Hence $\Lambda^*$ is often called the ground state representation (or the Brezis--V\'azquez transformation \cite{BrVa}). The following lemma shows that $\mathcal L$ is unitarily  equivalent to $-\Delta_{\R^2}$ on the space of radial functions through the map $\Lambda$, where $\Delta_{\R^2}$ denotes the Laplacian on $\R^2$. 

%{lemma}
\begin{lemma}
\label{lemma_5_1}
For any radial $f\in D(\mathcal L)$, $\mathcal Lf=-\Lambda^*\Delta_{\R^2}\Lambda f$. In particular, we have 
$$
e^{-it\mathcal L}f=\Lambda^*e^{it\Delta_{\R^2}}\Lambda f,\quad f\in L^2_r(\R^n).
$$
\end{lemma}

%proof
\begin{proof}
In the polar coordinate, a direct calculation shows that for any radial $f\in D_*$,
$$
\Lambda^*\left(\partial_r^2+\frac1r\partial_r\right)\Lambda f=\partial_r^2f+\frac{n-1}{r}\partial_rf+\frac{(n-2)^2}{4r^2}f=-\mathcal Lf.
$$
Since $D_*$ is a dense subset of $D(\mathcal L)$, the uniqueness of the Friedrichs extension implies the first assertion. The second assertion then follows from Stone's theorem (see \cite{ReSi_I}). 
\end{proof}

We next give a characterization of $\H_r^s$. Recall that the Sobolev space $\H^s$ adapted to $\mathcal L$ is the completion of $D_*$ with respect to the norm $\norm{\<\mathcal L\>^{s/2}f}_{L^2}$, where $\<\mathcal L\>^{s/2}$ is defined through the spectral decomposition theorem. More generally, for any $g\in L^2_{\mathrm{loc}}(\R)$, $g(\mathcal L)$ is defined by 
$$
g(\mathcal L)=\int g(\lambda) dE_{\mathcal L}(\lambda),\quad D(g(\mathcal L))=\left\{u\in L^2\ |\ \int|g(\lambda)|^2d\norm{E_{\mathcal L}(\lambda)u}_{L^2}^2<\infty\right\}
$$
where $dE_{\mathcal L}$ denotes the spectral measure associated with $\mathcal L$ (see \cite{ReSi_I}). 

%{lemma}
\begin{lemma}
\label{lemma_5_2}
Let $g\in L^2_{\loc}(\R)$. Then, for any radial $f\in D_*$, we have
\begin{align}
g(\mathcal L)f=\Lambda^*g(-\Delta_{\R^2})\Lambda f.
\end{align}
In particular, for any $s\in \R$, $\H_r^s=\Lambda^*H^s_r(\R^2)$ or equivalently $\Lambda \H_r^s=H^s_r(\R^2)$. 
\end{lemma}

%proof
\begin{proof}
Let $f\in L^2_r(\R^n)$. By Lemma \ref{lemma_5_1} and a direct computation, we have $(\mathcal L-z)^{-1}f=\Lambda^*(-\Delta_{\R^2}-z)^{-1}\Lambda f$ for any $z\in \C\setminus\R$ and $f\in L^2_r(\R^n)$. Combined with following Stone's formula
$$
E_{\mathcal L}(\lambda)=\frac{1}{2\pi i}\lim_{\ep \to0}\Big((\mathcal L-\lambda-i\ep)^{-1}-(\mathcal L-\lambda+i\ep)^{-1}\Big),
$$
this equality implies the unitary equivalence of the spectral measures: $$E_{\mathcal L}(\lambda)f=\Lambda^*E_{-\Delta_{\R^2}}(\lambda)\Lambda f.$$ Therefore, we know by the spectral decomposition theorem that, for any $v\in L^2_r(\R^n)$, 
$$
\<g(\mathcal L)f,v\>=\int_\R g(t)d\<E_\mathcal L(\lambda)f,v\>=\int_\R g(t)d\<E_{-\Delta_{\R^2}}(\lambda)\Lambda f,\Lambda v\>=\<\Lambda^*g(-\Delta_{\R^2})\Lambda f,v\>,
$$
which implies the desired assertion. 
\end{proof}

We prepare one more lemma which follows by direct computations.

%{lemma}
\begin{lemma}
\label{lemma_5_3}
Let $r=|x|$, $\tilde\alpha=\alpha$, $\tilde \beta=\beta-\alpha(n-2)/2$ and  $\tilde \lambda=c_n^\alpha\lambda$. Then $N(\Lambda z)=\Lambda \tilde N(z)$ for any $z\in \C$. Moreover $u_{\ap}(t)=\Lambda \tilde u_{\ap}(t)$ if $\varphi=i^{-\frac{n-2}{2}}\Lambda\tilde \varphi$.
\end{lemma}

We are now ready to show Theorem \ref{theorem_2}.

%proof
\begin{proof}[Proof of Theorem \ref{theorem_2}]
%Let $\F_{\R^2}$ and $\mathcal D_{\R^2}(t)$ be the two-dimensional Fourier transform and dilation defined  by \eqref{dilation} with $d=2$, respectively. By virtue of the formula $e^{it\Delta_{\R^2}}=\mathcal M\mathcal D_{\R^2}\F_{\R^2}\mathcal M$ and Lemma \ref{lemma_5_1}, $e^{-it\mathcal L}$ satisfies the following factorization formula\begin{align}\label{MDFM_2}e^{-it\mathcal L}f=\Lambda^*\mathcal M\mathcal D_{\R^2}\F_{\R^2}\mathcal M\Lambda f,\quad f\in L^2_r(\R^n).\end{align}As in case of \eqref{INLS} discussed in Appendix \ref{appendix_A} below, we then can reformulate (under the radial symmetry) \eqref{NLSI} subjected to the asymptotic condition \eqref{theorem_2_1} to the following integral equation\begin{align}\label{integral_equation_2}\tilde u(t)=\tilde u_{\mathrm{ap}}(t)+\tilde{\mathcal K}_1[\tilde u](t)+\tilde{\mathcal K}_2(t)+\tilde{\mathcal R}\tilde w(t),\end{align}where $\tilde {\mathcal R}=\Lambda^*\mathcal R_{\R^2}\Lambda=\Lambda^*\mathcal M\mathcal D_{\R^2}\F_{\R^2}(\mathcal M-1)\F_{\R^2}^{-1}\Lambda$ and \begin{align*}\tilde w(t)&=\tilde \varphi(x) \exp\left(-\frac{i}{2}\tilde \lambda |x|^{-\tilde \beta}|\tilde \varphi(x)|^{\tilde \alpha}\log t\right),\\\tilde{\mathcal K}_1[\tilde u](t)&=i\int_t^\infty e^{-i(t-s)\mathcal L}\Big(\tilde N(\tilde u(s))-\tilde N(\tilde u_{\ap}(s))\Big)ds,\\\tilde{\mathcal K}_2(t)&=-i\int_t^\infty e^{-i(t-s)\mathcal L}\tilde{\mathcal R}(s)\tilde N(\tilde w(s))\frac{ds}{2s}.\end{align*}
Suppose that the conditions in Theorem \ref{theorem_2} are fulfilled. We set $\alpha=\tilde \alpha$, $\beta=\tilde \beta+\alpha(n-2)/2$, $\lambda=c_n^{-\tilde\alpha}\tilde\lambda$ and $\varphi=i^{-\frac{n-2}{2}}\Lambda\tilde \varphi$. Then we know from Lemmas \ref{lemma_5_1} and \ref{lemma_5_3} that \eqref{INLS} is equivalent to \eqref{NLSI} under the radial symmetry in the sense that $\tilde u\in C(\R;L^2_r(\R^n))$ is a radial solution to \eqref{NLSI} if and only if $u=\Lambda \tilde u\in C(\R;L^2_r(\R^2))$ is a radial solution to \eqref{INLS} with $d=2$. We shall check that the assumptions in Theorem \ref{theorem_1} with $d=2$ are satisfied. A direct calculation yields that $\alpha,\beta,\delta$ and $\theta$ satisfy the desired conditions. By Lemma \ref{lemma_5_2}, $\varphi\in H^\delta_r$ and $\varphi$ satisfies Assumption \ref{assumption_A} since $r^{-\beta/\alpha}\varphi=i^{-\frac{n-2}{2}}c_n r^{-\tilde \beta/\tilde\alpha}\tilde \varphi\in L^\infty$ and 
$$
\norm{\<D\>^{\delta-1}r^{-1}\varphi}_{L^2(\R^2)}=\norm{\Lambda \<\mathcal L\>^{(\delta-1)/2}\Lambda^*r^{-1}\Lambda \tilde \varphi}_{L^2(\R^2)}=\norm{\<\mathcal L\>^{(\delta-1)/2}r^{-1}\tilde \varphi}_{L^2(\R^n)}<\infty.
$$
We also have 
$\norm{r^{-\beta/\alpha}\varphi}_{L^\infty}=c_n\norm{r^{-\tilde\beta/\tilde\alpha}\tilde \varphi}_{L^\infty}\le c_n\ep$.  
Hence, Theorem \ref{theorem_1} and the unitary equivalence between \eqref{INLS} and \eqref{NLSI} show that, for sufficiently small $\ep>0$, there exists a unique radial solution $\tilde u\in C(\R;L^2_r(\R^n))$ to \eqref{NLSI} satisfying \eqref{theorem_2_1}. This completes the proof. 
\end{proof}

%section
\section{Proof of Lemma \ref{lemma_w}}
\label{section_w}
This section is devoted to the proof of Lemma \ref{lemma_w}. In what follows, we write  for simplicity$$E_j=E_{j}(\varphi).$$We first record two known results on the fractional derivatives which will be used frequently in this section (see \cite[Theorem A.8]{KePoVe_CPAM} for Lemma \ref{lemma_chain} and \cite[Proposition A.1]{Visan} for Lemma \ref{lemma_fractional}). 

%{lemma}
\begin{lemma}
\label{lemma_chain}
Let $s>0$, $1<r<\infty$, $1<p_1,p_2,q_1,q_2\le \infty$ and $1/r=1/p_1+q_1=1/p_2+1/q_2$. Then we have the following fractional Leibniz rule: 
$$
\norm{|D|^s(fg)}_{L^r}\lesssim \norm{|D|^sf}_{L^{p_1}}\norm{g}_{L^{q_1}}+\norm{f}_{L^{p_2}}\norm{|D|^sg}_{L^{q_2}}.
$$
\end{lemma}

\begin{lemma}
\label{lemma_fractional}
Let $F$ be a H\"older continuous function of order  $0<\gamma<1$. Suppose that $0<\sigma<\gamma$, $1<p<\infty$ and $\sigma/\gamma<s<1$. Then
\begin{align}
\norm{|D|^\sigma F(f)}_{L^p}\lesssim \norm{f}_{L^{\infty}}^{\gamma-\sigma/s}\norm{|D|^sf}_{L^{p\sigma /s}}^{\sigma/s}.
\end{align}
\end{lemma}

For short we set
$
\psi(x):=|x|^{-\beta}|\varphi(x)|^{\alpha}
$
so that $N(\varphi)=\lambda \varphi \psi$ and
\begin{align}
\label{lemma_w_proof_1}
\nabla (\varphi e^{i\mu\psi})&=(\nabla\varphi +i\mu\varphi\nabla \psi)e^{i\mu\psi},\\
\label{lemma_w_proof_2}
\nabla (N(\varphi)e^{i\mu\psi})&=\lambda\Big((\nabla \varphi)\psi +\varphi\nabla\psi+i\mu\varphi\psi\nabla \psi\Big)e^{i\mu\psi}. 
\end{align}
Note that $\varphi\nabla \psi$ is given by
\begin{align}
\varphi\nabla\psi
%&=-\beta |x|^{-\beta}|\varphi|^\alpha |x|^{-2}x\varphi+\frac\alpha2|x|^{-\beta}|\varphi|^{\alpha}\nabla\varphi+\frac\alpha2|x|^{-\beta}|\varphi|^{\alpha-2}\varphi^2\overline{\nabla\varphi}\\
\label{lemma_w_proof_3}
=-\beta |x|^{-2}x\varphi \psi+\frac\alpha2\psi\nabla\varphi+\frac\alpha2|x|^{-\beta}|\varphi|^{\alpha-2}\varphi^2\overline{\nabla\varphi}\end{align}
which is well-defined for all $\alpha>0$ even when $\nabla \psi$ makes no sense for $\alpha<1$. 

We begin with the case $d=1$. 

%proof
\begin{proof}[Proof of Lemma \ref{lemma_w} in case of $d=1$] We first let $\delta=1$ and use \eqref{lemma_w_proof_1} and \eqref{lemma_w_proof_3} to find
\begin{align*}
%\label{lemma_w_proof_2}
\norm{\varphi e^{i\mu\psi}}_{H^1}
&\lesssim (\norm{\varphi}_{L^2}+\norm{\partial_x\varphi}_{L^2}+|\mu|\norm{\varphi\partial_x\psi}_{L^2})\norm{e^{i\mu\psi}}_{L^\infty}\\
&\lesssim\norm{\varphi}_{H^1}+|\mu|(\norm{\psi\partial_x\varphi}_{L^2}+|\beta|\norm{|x|^{-1}\varphi\psi}_{L^2})\\
%&\lesssim \norm{\varphi}_{H^1}+|\mu|\norm{\psi}_{L^\infty}(\norm{\varphi}_{H^1}+\norm{|x|^{-1}\varphi}_{L^2})\\
&\lesssim (\norm{\varphi}_{H^1}+|\beta|\norm{|x|^{-1}\varphi}_{L^2})(1+|\mu|\norm{\psi}_{L^\infty})\\
&\lesssim E_1E_2.
\end{align*}
Similarly, it follows from \eqref{lemma_w_proof_2} and \eqref{lemma_w_proof_3} that
\begin{align*}
\norm{N(\varphi)e^{i\mu\psi}}_{H^1}
&\lesssim  (\norm{\varphi\psi}_{L^2}+\norm{(\partial_x\varphi)\psi}_{L^2}+\norm{\varphi\partial_x\psi}_{L^2}+|\mu|\norm{\psi}_{L^\infty}\norm{\varphi\partial_x\psi}_{L^2}))\norm{e^{i\mu\psi}}_{L^\infty}\\
&\lesssim \norm{\varphi}_{H^1}\norm{\psi}_{L^\infty}+(1+|\mu|\norm{\psi}_{L^\infty})(\norm{\psi\partial_x\varphi}_{L^2}+|\beta|\norm{|x|^{-1}\varphi\psi}_{L^2})
\\
&\lesssim (\norm{\varphi}_{H^1}+|\beta|\norm{|x|^{-1}\varphi}_{L^2})\norm{\psi}_{L^\infty}(1+|\mu|\norm{\psi}_{L^\infty})\\
&\lesssim E_1E_2(1+|\mu|E_2).
\end{align*}

Suppose next $1/2<\delta<1$. Since $\alpha>1$, it is easy to see that for any $z_1,z_2\in \C$, 
\begin{align}
\label{lemma_w_proof_4_1}
||z_1|^\alpha-|z_2|^\alpha|&\lesssim (|z_1|^{\alpha-1}+|z_2|^{\alpha-1})|z_1-z_2|,\\
\label{lemma_w_proof_4_2}
|e^{i\mu |z_1|^\alpha}-e^{i\mu|z_2|^\alpha}|&\lesssim |\mu|(|z_1|^{\alpha-1}+|z_2|^{\alpha-1})|z_1-z_2|.
\end{align}
Then we use a characterization of the homogeneous Sobolev norm $\norm{\cdot}_{\dot H^\delta}$ via the Gagliardo semi-norm (see \cite[Proposition 3.4]{NPV}) and \eqref{lemma_w_proof_4_1} with $z_1=\psi(x)^{1/\alpha}$ and $z_2=\psi(y)^{1/\alpha}$ to find
\begin{align}
\nonumber
\norm{\psi}_{\dot H^\delta}
&\lesssim \left(\iint \frac{|\psi(x)-\psi(y)|^2}{|x-y|^{1+2\delta}}dxdy\right)^{1/2}\\
\nonumber
&\lesssim \norm{|x|^{-\beta/\alpha}\varphi}_{L^\infty}^{\alpha-1}\norm{|x|^{-\beta/\alpha}\varphi}_{\dot H^\delta}\\
\label{lemma_w_proof_4}&\lesssim E_2.
\end{align}
Using \eqref{lemma_w_proof_4_2} instead of \eqref{lemma_w_proof_4_1}, we similarly have
\begin{align}
\label{lemma_w_proof_5}
\norm{e^{i\mu\psi}}_{\dot H^\delta}
\lesssim |\mu|\norm{|x|^{-\beta/\alpha}\varphi}_{L^\infty}^{\alpha-1}\norm{|x|^{-\beta/\alpha}\varphi}_{\dot H^\delta}
\lesssim |\mu| E_2.
\end{align}
Then it follows from Lemma \ref{lemma_chain}, \eqref{lemma_w_proof_5} and  the embedding $H^\delta\subset L^\infty$ that
\begin{align*}
\norm{\varphi e^{i\mu\psi}}_{H^\delta}
&\lesssim \norm{\varphi}_{L^2}+\norm{\varphi}_{\dot H^\delta}\norm{e^{i\mu\psi}}_{L^\infty}+\norm{\varphi}_{L^\infty}\norm{e^{i\mu\psi}}_{\dot H^\delta}\\
%&\lesssim \norm{\varphi}_{H^\delta}(1+|\mu|\norm{|x|^{-\beta/\alpha}\varphi}_{L^\infty}^{\alpha-1}\norm{|x|^{-\beta/\alpha}\varphi}_{\dot H^\delta})\\
&\lesssim E_1(1+|\mu|E_2).
\end{align*}
Similarly, we know by Lemma \ref{lemma_chain}, \eqref{lemma_w_proof_4} and \eqref{lemma_w_proof_5} that
\begin{align*}
\norm{N(\varphi)e^{i\mu\psi}}_{H^\delta}
&\lesssim \norm{\varphi \psi}_{L^2}+\norm{\varphi}_{\dot H^\delta}\norm{\psi}_{L^\infty}+\norm{\varphi}_{L^\infty}(\norm{\psi}_{\dot H^\delta}+\norm{\psi}_{L^\infty}\norm{e^{i\mu\psi}}_{\dot H^\delta})\\
&\lesssim \norm{\varphi}_{H^\delta}(\norm{\psi}_{L^\infty}+\norm{\psi}_{\dot H^\delta}+\norm{\psi}_{L^\infty}\norm{e^{i\mu\psi}}_{\dot H^\delta})\\
&\lesssim E_1E_2(1+|\mu|E_2),
\end{align*}
which completes the proof of the lemma for the case $d=1$.
\end{proof}

%remark\begin{remark}For the term $|x|^{-1}\varphi\psi$, we can also prove, for any $\ep>0$$$\norm{|x|^{-1}\varphi\psi}_{L^2}\le \norm{\psi}_{L^\infty}\norm{\varphi}_{L^2(|x|\ge1)}+C_\ep\norm{|x|^{-(\beta/\alpha+1/2)/(\alpha+1)-\ep}\varphi}^{\alpha+1}_{L^2(|x\le1)}.$$ \end{remark}

The following lemma is the core of the proof in case of $d=2,3$. 

%{lemma}
\begin{lemma}
\label{lemma_psi}
Let $d=2,3$, $\beta$ and $\alpha$ satisfy \eqref{long_range_INLS}, $1<\delta<1+\alpha$ and $\mu\in \R$. Then we have
\begin{align}
\label{lemma_psi_1}
\norm{|D|^{\delta-1}\psi}_{L^{\frac{d}{\delta-1}}}
&\lesssim E_2,\\
\label{lemma_psi_2}
\norm{|D|^{\delta-1} e^{i\mu \psi}}_{L^{\frac{d}{\delta-1}}}
&\lesssim |\mu| E_2,\\
\label{lemma_psi_3}
\norm{|D|^{\delta-1} (\psi e^{i\mu \psi})}_{L^{\frac{d}{\delta-1}}}
&\lesssim E_2(1+|\mu|E_2),\\ 
\label{lemma_psi_4}
\norm{|D|^{\delta-1}(\varphi \nabla \psi)}_{L^2}
&\lesssim E_1E_2.
\end{align}
\end{lemma}

%proof
\begin{proof}
Note that $0<\alpha<1$ under the assumption. Since $F_1(z)=|z|^\alpha$ is $\alpha$-H\"older continuous by Lemma \ref{lemma_nonlinear}, we can apply Lemma \ref{lemma_fractional} to $\psi=|x|^{-\beta}|\varphi|^\alpha=F_1(|x|^{-\beta/\alpha}\varphi)$ obtaining
\begin{align}
\label{proof_lemma_psi_proof_1}
\norm{|D|^{\delta-1} \psi}_{L^{\frac{d}{\delta-1}}}
\lesssim \norm{|x|^{-\beta/\alpha}\varphi}_{L^\infty}^{\alpha-(\delta-1)/s}\norm{|D|^s(|x|^{-\beta/\alpha}\varphi)}_{L^{d/s}}^{(\delta-1)/s}
\end{align}
for any $s$ satisfying $(\delta-1)/\alpha<s<1$. For simplicity we choose such an exponent $s$ by the relation $2s=1+(\delta-1)/\alpha$. 
Since $d(1/d-s/d)=1-s$, Sobolev's inequality implies
\begin{align}
\nonumber
\norm{|D|^s(|x|^{-\beta/\alpha}\varphi)}_{L^{d/s}}
%&\lesssim \norm{|D|(|x|^{-\beta/\alpha}\varphi)}_{L^d}\\\nonumber
&\lesssim \norm{\<D\>(|x|^{-\beta/\alpha}\varphi)}_{L^d}\\
\nonumber
&\lesssim \norm{\nabla(|x|^{-\beta/\alpha}\varphi)}_{L^d}+\norm{|x|^{-\beta/\alpha}\varphi}_{L^d}\\
\label{proof_lemma_psi_proof_2}
&\lesssim |\beta|\norm{|x|^{-\beta/\alpha-1}\varphi}_{L^d}+\norm{|x|^{-\beta/\alpha}\nabla\varphi}_{L^d}+\norm{|x|^{-\beta/\alpha}\varphi}_{L^d},
\end{align}
where we have used the norm equivalence $W^{1,d}(\R^d)=\<D\>^{-1}L^d(\R^d)$. 
Since $0\le\beta/\alpha<1$, we can apply the following fractional Hardy inequality (see \cite{Herbst})
\begin{align*}
%\label{proof_lemma_psi_proof_4}
\norm{|x|^{-\beta/\alpha} f}_{L^d(\R^d)}\lesssim \norm{|D|^{\beta/\alpha} f}_{L^d(\R^d)}
\end{align*}
and the Sobolev embedding $\dot H^{d/2-1}\subset L^d$ to obtain 
\begin{align*}
&\norm{|x|^{-\beta/\alpha-1}\varphi}_{L^d}\lesssim \norm{|D|^{\beta/\alpha}(|x|^{-1}\varphi)}_{L^d}\lesssim \norm{|x|^{-1}\varphi}_{\dot H^{d/2+\beta/\alpha-1}}\lesssim \norm{|x|^{-1}\varphi}_{H^{\delta-1}},\\
&\norm{|x|^{-\beta/\alpha}\nabla\varphi}_{L^d}+\norm{|x|^{-\beta/\alpha}\varphi}_{L^d}\lesssim \norm{\varphi}_{H^{d/2+\beta/\alpha}}\lesssim \norm{\varphi}_{H^\delta}.
\end{align*}
These two estimates, combined with \eqref{proof_lemma_psi_proof_1} and \eqref{proof_lemma_psi_proof_2}, imply \eqref{lemma_psi_1}. 

The proof of \eqref{lemma_psi_2} is almost identical to that of \eqref{lemma_psi_1}. Indeed, since $F_2(z)=e^{i\mu|z|^\alpha}$ satisfies $|F_2(z_1)-F_2(z_2)|\lesssim |\mu||z_1-z_2|^\alpha$ by Lemma \ref{lemma_nonlinear}, Lemma \ref{lemma_chain} with $e^{i\mu\psi}=F_2(|x|^{-\beta/\alpha}\varphi)$ shows
$$
\norm{|D|^{\delta-1} e^{i\mu\psi}}_{L^{\frac{d}{\delta-1}}}
\lesssim |\mu| \norm{|x|^{-\beta/\alpha}\varphi}_{L^\infty}^{\alpha-(\delta-1)/s}\norm{|D|^s(|x|^{-\beta/\alpha}\varphi)}_{L^{d/s}}^{(\delta-1)/s}\lesssim |\mu|E_2.
$$

The estimate \eqref{lemma_psi_3} follows from \eqref{lemma_psi_1} and \eqref{lemma_psi_2}. Indeed, Lemma \ref{lemma_chain} implies
\begin{align*}
\norm{|D|^{\delta-1} (\psi e^{i\mu \psi})}_{L^{\frac{d}{\delta-1}}}
&\lesssim \norm{|D|^{\delta-1}\psi}_{L^{\frac{d}{\delta-1}}}\norm{e^{i\mu \psi}}_{L^\infty}+\norm{\psi}_{L^\infty}\norm{|D|^{\delta-1} e^{i\mu \psi}}_{L^{\frac{d}{\delta-1}}}\\
%&\lesssim \norm{|D|^{\delta-1}\psi}_{L^{\frac{d}{\delta-1}}}+\norm{|x|^{-\beta/\alpha}\varphi}_{L^\infty}^\alpha \norm{|D|^{\delta-1} e^{i\mu \psi}}_{L^{\frac{d}{\delta-1}}}\\
&\lesssim E_2(1+|\mu|E_2).
\end{align*}

To show \eqref{lemma_psi_4}, we  deal with the three terms $\beta\psi|x|^{-2}x\varphi$, $
\psi\nabla\varphi$ and $
|x|^{-\beta}|\varphi|^{\alpha-2}\varphi^2\overline{\nabla\varphi}$ in \eqref{lemma_w_proof_3} separately. For the first term $\beta\psi|x|^{-2}x\varphi$, we use Lemma \ref{lemma_chain} and \eqref{lemma_psi_1} to find
\begin{align*}
&\norm{|D|^{\delta-1}(\psi|x|^{-2}x\varphi)}_{L^2}\\
&\lesssim \norm{|D|^{\delta-1}\psi}_{L^{\frac{d}{\delta-1}}}\norm{|x|^{-2}x\varphi}_{L^{\frac{2d}{d-2(\delta-1)}}}+\norm{\psi}_{L^\infty}\norm{|D|^{\delta-1}(|x|^{-2}x\varphi)}_{L^2}\\
&\lesssim E_2\norm{|D|^{\delta-1}(|x|^{-2}x\varphi)}_{L^2},
\end{align*}
where we have also used $\dot H^{\delta-1}\subset L^{\frac{2d}{d-2(\delta-1)}}$. 
For each $j=1,...,d$, we can write
$$
|x|^{-2}x_j\varphi=\F\F^{-1} \Big(|x|^{-1}x_j\F\F^{-1}(|x|^{-1}\varphi)\Big)=-i\F R_j\F^{-1}(|x|^{-1}\varphi),
$$
where $R_j=i\F^{-1}|x|^{-1}x_j\F$ is the $j$-the direction Riesz transform (in the Fourier side). It is known that $R_j$ is a Calder\'on--Zygmund operator (see \cite{CoFe}). Moreover, $|\xi|^{2(\delta-1)}$ is a Muckenhoupt $A_2$-weight since $0<\delta-1<\alpha<d/2$ if $d=2,3$. The Coifman--Fefferman theorem \cite{CoFe} then shows
$$
\norm{|\xi|^{\delta-1}R_jf}_{L^2(\R^d_\xi)}\lesssim \norm{|\xi|^{\delta-1}f}_{L^2(\R^d_\xi)}.
$$
Combining this estimate with the above computation and the Plancherel theorem, we obtain
$$
\norm{|D|^{\delta-1}(|x|^{-2}x_j\varphi)}_{L^2}\lesssim \norm{|\xi|^{\delta-1}\F^{-1}(|x|^{-1}\varphi)}_{L^2}\lesssim \norm{|D|^{\delta-1}(|x|^{-1}\varphi)}_{L^2}
$$
for each $j=1,...,d$. Therefore,
\begin{align}
\label{proof_lemma_psi_proof_6}
\norm{|D|^{\delta-1}(\beta\psi|x|^{-2}x\varphi)}_{L^2}\lesssim E_1E_2.
\end{align}
For the second term $
\psi\nabla\varphi$, we know by Lemma \ref{lemma_chain} and \eqref{lemma_psi_1} that
\begin{align}
\nonumber
\norm{|D|^{\delta-1}(\psi \nabla\varphi)}_{L^2}
&\lesssim \norm{|D|^{\delta-1}\psi}_{L^{\frac{d}{\delta-1}}}\norm{\nabla\varphi}_{L^{\frac{2d}{d-2(\delta-1)}}}+\norm{\psi}_{L^\infty}\norm{|D|^{\delta-1}\nabla\varphi}_{L^2}\\
\label{proof_lemma_psi_proof_7}
&\lesssim E_1E_2.
\end{align}
For the last term $
|x|^{-\beta}|\varphi|^{\alpha-2}\varphi^2\overline{\nabla\varphi}$, we set $\widetilde\psi:=|x|^{-\beta}|\varphi|^{\alpha-2}\varphi^2$ and $F_3(z)=|z|^{\alpha-2}z^2$. Since $F_3$ is $\alpha$-H\"older continuous by Lemma \ref{lemma_nonlinear} and $\widetilde \psi=F_3(|x|^{-\beta/\alpha}\varphi)$, the same argument as above  shows
\begin{align}
\norm{|D|^{\delta-1}(\widetilde\psi\overline{\nabla\varphi})}_{L^2}
%\lesssim \norm{|D|^{\delta-1}\widetilde\psi}_{L^{\frac{d}{\delta-1}}}\norm{\overline{\nabla\varphi}}_{L^{\frac{2d}{d-2(\delta-1)}}}+\norm{\widetilde\psi}_{L^\infty}\norm{|D|^{\delta-1}\overline{\nabla\varphi}}_{L^2}
\label{proof_lemma_psi_proof_8}\lesssim E_1E_2.
\end{align}
Th estimate \eqref{lemma_psi_4} now follows from \eqref{proof_lemma_psi_proof_6}--\eqref{proof_lemma_psi_proof_8}. This completes the proof of the lemma. 
\end{proof}

We are now ready to show Lemma \ref{lemma_w} in case of $d=2,3$. 
%proof
\begin{proof}[Proof of Lemma \ref{lemma_w} in case of $d=2,3$]
We first observe that $\norm{f}_{H^\delta}\lesssim \norm{f}_{L^2}+\norm{f}_{\dot H^\delta}$ and 
\begin{align}
\label{lemma_w_proof_6}
\norm{\varphi e^{i\mu\psi}}_{L^2}\le E_1,\quad
\norm{N(\varphi)e^{i\mu\psi}}_{L^2}\le \norm{\varphi}_{L^2}\norm{\psi}_{L^\infty}\le E_1E_2.
\end{align}
To deal with the $\dot H^\delta$-norms of $\varphi e^{i\mu\psi}$ and $N(\varphi)e^{i\mu\psi}$, we next apply Lemma \ref{lemma_chain} to the right hand side in \eqref{lemma_w_proof_1}  and use the embedding $\dot H^{\delta-1}\subset L^{\frac{2d}{d-2(\delta-1)}}$, \eqref{lemma_psi_2} and \eqref{lemma_psi_4} to find
\begin{align}
\nonumber
\norm{\varphi e^{i\mu\psi}}_{\dot H^\delta}
&\lesssim \norm{|D|^{\delta-1}\nabla\varphi}_{L^2}\norm{e^{i\mu\psi}}_{L^\infty}+\norm{\nabla \varphi}_{L^{\frac{2d}{d-2(\delta-1)}}}\norm{|D|^{\delta-1} e^{i\mu\psi}}_{L^{\frac{d}{\delta-1}}}\\
\nonumber
&\quad\quad+|\mu|\norm{|D|^{\delta-1}(\varphi\nabla\psi)}_{L^2}\norm{e^{i\mu\psi}}_{L^\infty}+|\mu|\norm{\varphi\nabla\psi}_{L^{\frac{2d}{d-2(\delta-1)}}}\norm{|D|^{\delta-1} e^{i\mu\psi}}_{L^{\frac{d}{\delta-1}}}\\
\nonumber
&\lesssim \Big(\norm{\varphi}_{H^{\delta}}+|\mu|\norm{|D|^{\delta-1}(\varphi\nabla\psi)}_{L^2}\Big)\Big(1+\norm{|D|^{\delta-1} e^{i\mu\psi}}_{L^{\frac{d}{\delta-1}}}\Big)\\
\nonumber
&\lesssim (E_1+|\mu|E_1E_2)(1+|\mu| E_2)\\
\label{lemma_w_proof_7}
&\lesssim E_1(1+|\mu| E_2)^2. 
\end{align}
Similarly, we obtain by means of \eqref{lemma_w_proof_2} and Lemma \ref{lemma_chain} that
\begin{align*}
&\norm{N(\varphi)e^{i\mu\psi}}_{\dot H^{\delta}}\\
&\lesssim \norm{|D|^{\delta-1}((\nabla\varphi)\psi e^{i\mu\psi})}_{L^2}
+\norm{|D|^{\delta-1}(\varphi(\nabla\psi) e^{i\mu\psi})}_{L^2}
+|\mu|\norm{|D|^{\delta-1}(\varphi\psi (\nabla\psi) e^{i\mu\psi})}_{L^2}\\
&\lesssim \norm{|D|^{\delta-1}\nabla \varphi}_{L^2}\norm{\psi e^{i\mu\psi}}_{L^\infty}+\norm{\nabla\varphi}_{L^{\frac{2d}{d-2(\delta-1)}}}\norm{|D|^{\delta-1}(\psi e^{i\mu\psi})}_{L^{\frac{d}{\delta-1}}}\\
&\quad\quad+\norm{|D|^{\delta-1}(\varphi\nabla\psi)}_{L^2}\norm{e^{i\mu\psi}}_{L^\infty}+\norm{\varphi\nabla\psi}_{L^{\frac{2d}{d-2(\delta-1)}}}\norm{|D|^{\delta-1} e^{i\mu\psi}}_{L^{\frac{d}{\delta-1}}}\\
&\quad\quad+|\mu|\norm{|D|^{\delta-1}(\varphi\nabla\psi)}_{L^2}\norm{\psi e^{i\mu\psi}}_{L^\infty}+|\mu|\norm{\varphi\nabla\psi}_{L^{\frac{2d}{d-2(\delta-1)}}}\norm{|D|^{\delta-1}(\psi e^{i\mu\psi})}_{L^{\frac{d}{\delta-1}}}\\
&\lesssim \norm{\varphi}_{H^{\delta}}\Big(\norm{\psi}_{L^\infty}+\norm{|D|^{\delta-1}(\psi e^{i\mu\psi})}_{L^{\frac{d}{\delta-1}}}\Big)\\
&\quad\quad+\norm{|D|^{\delta-1}(\varphi\nabla\psi)}_{L^2}\Big(1+\norm{|D|^{\delta-1}(\psi e^{i\mu\psi})}_{L^{\frac{d}{\delta-1}}}\Big)\\
&\quad\quad+|\mu|\norm{|D|^{\delta-1}(\varphi\nabla\psi)}_{L^2}\Big(\norm{\psi}_{L^\infty}+\norm{|D|^{\delta-1}(\psi e^{i\mu\psi})}_{L^{\frac{d}{\delta-1}}}\Big).
\end{align*}
This estimate, together with \eqref{lemma_psi_3} and \eqref{lemma_psi_4}, shows
\begin{align}
\nonumber
\norm{N(\varphi)e^{i\mu\psi}}_{\dot H^{\delta}}
&\lesssim E_1\{E_2+E_2(1+|\mu|E_2)\}+E_1E_2\{1+E_2(1+|\mu|E_2)\}\\
\nonumber
&\quad\quad+|\mu|E_1E_2\{E_2+E_2(1+|\mu|E_2)\}
\\
\nonumber
&\lesssim E_1E_2(1+(1+|\mu|)E_2+|\mu|E_2^2+|\mu|^2E_2^2)\\
\label{lemma_w_proof_8}
&\lesssim E_1E_2(1+\<\mu\>E_2)^2. 
\end{align}
The lemma for the case $d=2,3$ now follows from \eqref{lemma_w_proof_6}, \eqref{lemma_w_proof_7} and \eqref{lemma_w_proof_8}. 
\end{proof}

\section{Proof of Theorem \ref{theorem_C_1}}
\label{section_C}
Here we prove Theorem \ref{theorem_C_1}. The global well-posedness of  \eqref{INLS} in $H^1$ was proved by \cite{GeSt} (see also \cite{Dinh}). Moreover, if $u(0)=u_0$ belong to $\Sigma$, then so does the solution $u$. Indeed, taking the estimate \eqref{theorem_C_1_proof_5} below into account, one can show $u\in C(\R;\Sigma)$ by the completely same argument as that in  \cite[Lemma 6.5.2]{Cazenave}. This completes the part of the existence of the global solution in $\Sigma$. 

The proof of the scattering \eqref{theorem_C_1_2} is based on the method by Tsutsumi--Yajima \cite{TsYa}. By the decomposition \eqref{MDFM} and the fact $\mathcal M(t)\to1$ strongly in $L^2$ as $\to\pm\infty$, \eqref{theorem_C_1_2} holds if and only if
\begin{align*}
\norm{\mathcal D(t)^{-1}\mathcal M(t)^{-1}u(t)-\mathcal Fu_\pm}_{L^2}\to0,\quad t\to \pm\infty.
\end{align*}
We may consider the forward case $t\to+\infty$ only without loss of generality. Let $\mathcal T(t)f(t):=f(t^{-1}/4)$ and let $v$ be the pseudo-conformal transform of $u$:
\begin{align}
\label{PCT}
v(t,x):=\overline{[\mathcal T(t)\mathcal D(t)^{-1}\mathcal M(t)^{-1}u](t,x)}=\frac{1}{(2it)^{d/2}}e^{\frac{i|x|^2}{4t}}\overline{u}\Big(\frac{1}{4t},\frac{x}{2t}\Big),\quad t>0.
\end{align}
Then it is enough to show $v(t)$ converges strongly in $L^2$ as $t\to +0$. To this end, we first observe $v\in C((0,\infty);\Sigma)$ since the map \eqref{PCT} leaves $\Sigma$ invariant. Note that $H^1$ is not invariant under \eqref{PCT}, which is the reason why we consider the $\Sigma$-solution. On the other hand, the $L^2$-norm is also conserved under \eqref{PCT}. With the $L^2$-conservation law for \eqref{INLS}, we thus have
\begin{align}
\label{theorem_C_1_proof_1}
\norm{v(t)}_{L^2}=\norm{v(1)}_{L^2},\quad 0<t\le1.
\end{align}
Moreover, a direct calculation yields that $v$ solves
\begin{align}
\label{TINLS}
i\partial_t v+\Delta v=|2t|^{\sigma}N(v),\quad \sigma:=\alpha d/2+\beta-2\in (-1,0],
\end{align}
Recalling $N(v)=\lambda |x|^{-\beta}|v|^\alpha v$, we next define an energy $E(t)$  for \eqref{TINLS} by
$$
E(t)=(2t)^{-\sigma} \frac12\norm{\nabla v(t)}_{L^2}^2+\frac{\lambda}{\alpha+2}\int |x|^{-\beta}|v(t)|^{\alpha+2}dx.
$$
Note that $E(1)$ satisfies
\begin{align}
\label{theorem_C_1_proof_2}
E(1)\lesssim \norm{v(1)}_{H^1}^2+\norm{v(1)}_{H^1}^{\alpha+2}
\end{align}
 Indeed, H\"older's inequality \eqref{Holder} implies
$$
\left(\int |x|^{-\beta}|v|^{\alpha+2}dx\right)^{\frac{1}{\alpha+2}}
=\norm{|x|^{-\frac{\beta}{\alpha+2}}v}_{L^{\alpha+2}}
\lesssim \norm{|x|^{-\frac{\beta}{\alpha+2}}}_{L^{\frac{d(\alpha+2)}{\beta},\infty}}\norm{v}_{L^{q,\alpha+2}}\lesssim \norm{v}_{L^{q,2}},
$$
where $q>\alpha+2$ is defined by the relation $(\alpha+2)/q=1-{\beta}/{d}$ and the continuous embedding $L^{q,2}\subset L^{q,\alpha+2}$ was used. It is easy to check that $d(1/2-1/q)\le1$ if ${\alpha d}/{2}+\beta\le \alpha+2$, which is satisfied under \eqref{theorem_C_1_1}. %\begin{color}{red}(\UTF{0082}±\UTF{0082}\UTF{00CC}\UTF{008D}\UTF{0080}\UTF{0082}\UTF{00CD}\UTF{008D}\UTF{00C5}\UTF{008F}I\UTF{0093}I\UTF{0082}\UTF{00C9}\UTF{0082}\UTF{00CD}\UTF{008F}\UTF{00C1}\UTF{0082}\UTF{00B7}\UTF{0081}j\UTF{008E}\UTF{00C0}\UTF{008D}\UTF{00DB},\begin{align*}&d(1/2-1/q)\le1\ \LRA\ 1/2-1/q\le 1/d\ \LRA\ \frac{1}{\alpha+2}\left(1-\frac{\beta}{d}\right)=1/q\ge1/2-1/d=\frac{d-2}{2d}\\&\LRA\ 2d(1-\beta/d)\ge (d-2)(\alpha+2)\ \LRA\ 2d-2\beta\ge \alpha d-2\alpha+2d-4\\&\LRA\ \alpha d+2\beta\le 2\alpha+4\ \LRA\ \frac{\alpha d}{2}+\beta\le \alpha+2\end{align*}\end{color}
Sobolev's inequality then shows $\norm{v}_{L^{q,2}}\lesssim \norm{v}_{H^1}$ and hence \eqref{theorem_C_1_proof_2} follows. We next multiply \eqref{TINLS} by $(2t)^{-\sigma}\partial_t \overline{v}$, integrate in $x$ and then take the real part to obtain
\begin{align}
\label{theorem_C_1_proof_3}
\frac{d}{dt}E(t)=-\sigma (2t)^{-\sigma-1}E(t)\ge0,\quad 0<t\le1,
\end{align}
where the last condition in \eqref{theorem_C_1_1} was used to ensure $-\sigma\ge0$. Since $\lambda>0$, by using \eqref{theorem_C_1_proof_1}, \eqref{theorem_C_1_proof_2} and \eqref{theorem_C_1_proof_3}, we obtain for $0<t\le 1$ that
\begin{align}
\label{theorem_C_1_proof_4}
\norm{v(t)}_{L^2}\lesssim1,\quad \norm{\nabla v(t)}_{L^2}\lesssim t^{\sigma/2},\quad \norm{|x|^{-\frac{\beta}{\alpha+2}}v}_{L^{\alpha+2}}\lesssim1,
\end{align}
where the implicit constants depend on $\lambda,\alpha,\norm{v(1)}_{H^1}$, but are independent of $t$. 

Now, given a test function $\varphi\in H^1$ and $0<s<t\le1$, we compute
\begin{align*}
\<v(t)-v(s),\varphi\>
%&=\int_s^t\<\partial_\tau v(\tau),\varphi>d\tau\\
=-i\int_s^t\<\nabla v(\tau),\nabla\varphi\>d\tau-i\int_s^t |2\tau|^{\sigma}\<N(v(\tau)),\varphi\>d\tau.
\end{align*}
By the fact $\sigma>-1$, the above bounds \eqref{theorem_C_1_proof_4} and the estimate
\begin{align}
\label{theorem_C_1_proof_5}
|\<N(v),\varphi\>|
%\lesssim \norm{|x|^{-\frac{\beta(\alpha+1)}{\alpha+2}}|v|^{\alpha+1}}_{L^{\frac{\alpha+2}{\alpha+1}}}\norm{|x|^{-\frac{\beta}{\alpha+2}}\varphi}_{L^{\alpha+2}}
\lesssim \norm{|x|^{-\frac{\beta}{\alpha+2}}v}_{L^{\alpha+2}}^{{\alpha+1}}\norm{|x|^{-\frac{\beta}{\alpha+2}}\varphi}_{L^{\alpha+2}}
\lesssim\norm{\varphi}_{H^1},
\end{align}
the weak limit $\ds v(0):=\wlim_{t\to+0}v(t)$ in $L^2$ exists. Moreover, plugging $\varphi=v(t)$ into the above equation with $s=0$ and using \eqref{theorem_C_1_proof_4}, we have
$$
|\<v(t)-v(0),v(t)\>|\lesssim t^{\sigma/2}\int_0^t \tau^{\sigma/2}d\tau +\int_0^t\tau^\sigma d\tau\lesssim  t^{\sigma+1}.
$$
Therefore, together with the weak-convergence of $v$ proved just above, we find
\begin{align*}
\norm{v(t)-v(0)}^2_{L^2}
&\le |\<v(t)-v(0),v(t)\>|+|\<v(t)-v(0),v(0)\>|\\
&\lesssim t^{\sigma+1}+|\<v(t)-v(0),v(0)\>|\to 0
\end{align*}
as $t \to+0$ since $\sigma>-1$. This completes the proof of Theorem \ref{theorem_C_1}.

%appendix
\appendix

%section
\section{Derivation of the integral equation}
\label{appendix_A}
Here we derive the integral equation \eqref{integral_equation} for a solution $u$ to \eqref{INLS} subjected to the asymptotic condition \eqref{theorem_1_2}. It follows from the equation \eqref{INLS} that
\begin{align*}
i\partial_t\left(\mathcal F e^{-it\Delta}u\right)=\mathcal F e^{-it\Delta}N(u),
\end{align*}
while a direct calculation shows
$$
i\partial_t w=\frac{1}{2t}N(w).
$$
These two equations, combined with the factorization formula \eqref{MDFM}, imply
\begin{align}
\nonumber
i\partial_t\left(\mathcal F e^{-it\Delta}u-w\right)
\nonumber
%&=\mathcal F e^{-it\Delta}\left(N(u)-\frac{1}{2t}e^{it\Delta}\mathcal F^{-1}N(w)\right)\\\nonumber
&=\mathcal F e^{-it\Delta}\left(N(u)-\frac{1}{2t}\mathcal M\mathcal D\mathcal F\mathcal M\mathcal F^{-1}N(w)\right)\\
\label{appendix_A_1}
&=\mathcal F e^{-it\Delta}\left(N(u)-\frac{1}{2t}\mathcal M\mathcal D N(w)\right)-\frac{1}{2t}\mathcal F e^{-it\Delta}\mathcal R N(w),
\end{align}
where, by virtue of \eqref{long_range_INLS} and \eqref{dilation}, the second term of the right hand side is written in the form
\begin{align}
\nonumber
\frac{1}{2t}\mathcal M\mathcal D N(w)
&=\frac{1}{2t}\mathcal M(t) \frac{1}{(2it)^{d/2}}\left|\frac{x}{2t}\right|^{-\beta} \lambda \left|w\left(\frac{x}{2t}\right)\right|^\alpha w\left(\frac{x}{2t}\right)\\
\nonumber
&=\frac{1}{(2t)^{1-\beta-\alpha d/2}} |x|^{-\beta}\lambda \left|\mathcal M(t) \frac{1}{(2it)^{d/2}}w\left(\frac{x}{2t}\right)\right|^\alpha \mathcal M(t)\frac{1}{(2it)^{d/2}}w\left(\frac{x}{2t}\right)\\
\nonumber
&=N(\mathcal M\mathcal D w)\\
\label{appendix_A_2}
&=N(u_{\mathrm{ap}}).
\end{align}
Moreover, $w(t)$ can be written in the form
\begin{align}
\label{appendix_A_3}
w(t)=\mathcal F e^{-it\Delta} \mathcal M\mathcal D\mathcal F\mathcal M\mathcal F^{-1} w=\mathcal F e^{-it\Delta} (u_{\mathrm{ap}}+\mathcal Rw).
\end{align}
\eqref{appendix_A_1}--\eqref{appendix_A_3} then imply
\begin{align*}
i\partial_t\mathcal F e^{-it\Delta}\left(u-u_{\mathrm{ap}}-\mathcal Rw\right)=\mathcal F e^{-it\Delta}\Big(N(u)-N(u_{\mathrm{ap}})\Big)-\frac{1}{2t}\mathcal F e^{-it\Delta}\mathcal RN(w).
\end{align*}
With the asymptotic condition \eqref{theorem_1_2} at hand, we obtain the integral equation \eqref{integral_equation} from this equation by integrating over $[t,\infty)$ and multiplying $e^{it\Delta}\F^{-1}$. 

%section
\section{Global existence for Cauchy problem}
\label{appendix_B}
Here we prove the global existence of the $L^2$-solutions to the Cauchy problem for \eqref{INLS} in the mass subcritical cases, including the case satisfying \eqref{long_range_INLS}. The following proposition is due to \cite[Theorems 1.2 and 1.7]{Guzman}. For the reader's convenience as well as for the sake of self-containedness, we give details of the proof.

%proposition
\begin{proposition}
\label{proposition_B_1}
Suppose $d\ge1$, $0<\beta<\min(d,2)$ and $0<\alpha<(4-2\beta)/d$ and set 
$$
p_3=\frac{4(\alpha+2)}{\alpha d+2\beta},\quad q_3=\frac{\alpha+2}{1-\beta/d}.
$$
Then, for any $t_0\in \R$ and $u_0\in L^2(\R^d)$, there exists a unique global solution $u\in C(\R;L^2)\cap L^{p_3}_{\loc}(\R;L^{q_3,2}(\R^d))$ to \eqref{INLS} with the initial condition $u(t_0)=u_0$ such that
\begin{align*}
u(t+t_0)=e^{it\Delta}u_0-i\int_0^te^{i(t-s)\Delta}N(u(s+t_0))ds.
\end{align*}
Moreover, $u\in L^p_{\loc}(\R;L^q(\R^d))$ for any admissible pair $(p,q)$. 
\end{proposition}

%proof
\begin{proof}The proof follows basically the same line as that in Tsutsumi \cite{Tsutsumi_FE}, except for the proof of the $L^2$-conservation law for which we will use the method by Ozawa \cite{Ozawa_CVPDE}. 

We may assume $t_0=0$, the general case being similar. Consider the map
$$
\Phi(v)=e^{it\Delta}u_0-i\int_0^t e^{i(t-s)\Delta}N(v(s))ds.
$$
Let $\norm{v}_{\mathscr Z_T}=\norm{v}_{L^\infty_TL^2_x}+\norm{v}_{L^{p_3}_TL^{q_3,2}_x}$, where  $L^p_TL^q_x:=L^p([-T,T];L^q(\R^d))$. Note that $(p_3,q_3)$ is an admissible pair such that
$
(\alpha+1)/q_3=1/q_3'-\beta/d
$. 
Lemma \ref{lemma_Strichartz} then implies
\begin{align*}
\norm{\Phi(v)}_{L^\infty_T L^2_x}+\norm{\Phi(v)}_{L^{p_3}_TL^{q_3,2}_x}+\norm{\Phi(v)}_{L^p_TL^{q}_x}\lesssim \norm{u_0}_{L^2}+|\lambda|\norm{|x|^{-\beta}|v|^{\alpha+1}}_{L^{p_3'}_TL^{q_3',2}_x}
\end{align*}
for any admissible pair $(p,q)$. It follows from Lemma \ref{lemma_Holder} that the last term satisfies
\begin{align*}
\norm{|x|^{-\beta}|v|^{\alpha+1}}_{L^{p_3'}_TL^{q_3',2}_x}
&\lesssim \norm{|x|^{-\beta}}_{L^{d/\beta,\infty}_x}\norm{|v|^{\alpha+1}}_{L^{p_3'}_TL^{r_1,2}_x}\\
%&\lesssim \norm{v}_{L^{p_3'(\alpha+1)}_TL^{r_1(\alpha+1),2(\alpha+1)}_x}^{\alpha+1}\\
&\lesssim \norm{v}_{L^{p_3'(\alpha+1)}_TL^{q_3,2(\alpha+1)}_x}^{\alpha+1}\\
&\lesssim T^{\delta}\norm{v}_{L^{p_3}_TL^{q_3,2}_x}^{\alpha+1},
\end{align*}
where $r_1$ is defined by the relation $1/r_1=1/q_3'-\beta/d$ so that $r_1(\alpha+1)=q_3$ and $\delta=1-(\alpha+2)/p_3$. Note that the condition $0<\alpha<(4-2\beta)/d$ implies $(\alpha+2)/p_3<1$, so $\delta>0$. $\Phi$ thus satisfies 
$$
\norm{\Phi(v)}_{\mathscr Z_T}\lesssim \norm{u_0}_{L^2_x}+T^\delta \norm{v}_{\mathscr Z_T}^{\alpha+1}.
$$
Since $||z_1|^\alpha z_1-|z_2|^\alpha z_2|\lesssim (|z_1|^\alpha+|z_2|^\alpha)|z_1-z_2|$, a similar computation implies
$$
\norm{\Phi(v_1)-\Phi(v_2)}_{\mathscr Z_T}\lesssim T^\delta (\norm{v_1}_{\mathscr Z_T}^\alpha+\norm{v_2}_{\mathscr Z_T}^\alpha)\norm{v_1-v_2}_{\mathscr Z_T}.
$$
By the contraction mapping theorem, we therefore obtain the local solution $u\in C_TL^2_x\cap L^{p_3}_TL^{q_3,2}_x$ for sufficiently small $T=T(\alpha,\beta,d,\norm{u_0}_{L^2})$, which belongs to $L^{p}_TL^{q}_x$ for any admissible pair $(p,q)$. 

To show the uniqueness, we take two solutions $u_1$ and $u_2$ with the same initial datum and set $T'=\sup\{t\in [0,T]\ |\ u_1(s)=u_2(s)\ \text{on}\ [-t,t]\}$. Assume $T'<T$ for contradiction. Choosing $T'<T''<T$ and setting $J=[-T'',T']\cup[T',T'']$, we find by a similar argument as above that
\begin{align*}
&\norm{u_1-u_2}_{L^{p_3}(J;L^{q_3,2}_x)}\\
&\le C |T''-T'|^\delta\left(\norm{u_1}_{L^{p_3}(J;L^{q_3,2}_x)}^\alpha+\norm{u_2}_{L^{p_3}(J;L^{q_3,2}_x)}^\alpha\right)\norm{u_1-u_2}_{L^{p_3}(J;L^{q_3,2}_x)}
\end{align*}
with some $C>0$ independent of $T',T'',u_1$ and $u_2$. Taking $T''$ sufficiently close to $T'$ implies that the right hand side is dominated by $\norm{u_1-u_2}_{L^{p_3}(J;L^{q_3,2}_x)}/2$ and hence $u_1(t)=u_2(t)$ on $J$. This contradicts with the definition of $T'$, yielding $T'=T$. 

To extend $u$ globally in time, it is enough to observe the following $L^2$-conservation law holds: 
\begin{align}
\label{proposition_appendix_B_1}
\norm{u(t)}_{L^2}=\norm{u_0}_{L^2},\quad t\in [-T,T].
\end{align}
Indeed, since $e^{it\Delta}u_0\in L^\infty_tL^2_x\cap L^{p_3}_TL^{q_3}_x$ by Lemma \ref{lemma_Strichartz} and $N(u)\in L^1_tL^2_x\cap L^{p_3'}_TL^{q_3'}_x$ by the above argument, for all $t\in [-T,T]$, the quantity $$\int_0^t \<e^{is\Delta}u_0,N(u(s))\>ds$$ makes sense as the duality coupling on $(L^\infty_tL^2_x\cap L^{p_3}_TL^{q_3}_x)\times (L^1_tL^2_x+L^{p_3'}_TL^{q_3'}_x)$ and is finite. Thus, \begin{align}\nonumber\norm{u(t)}_{L^2}^2&=\norm{e^{-it\Delta}u(t)}_{L^2}^2\\\label{proposition_appendix_B_2}&=\norm{u_0}_{L^2}^2-2\Im \int_0^t\<e^{is\Delta}u_0,N(u(s))\>ds+\bignorm{\int_0^t e^{-is\Delta}N(u(s))ds}_{L^2}^2.
\end{align}
By a similar duality argument, the following computations are also rigorously justified: 
\begin{align*}
\bignorm{\int_0^t e^{-is\Delta}N(u(s))ds}_{L^2}^2
&=\Re\int_0^t\int_0^t\<N(u(s)),e^{i(s-s')\Delta}N(u(s'))\>ds'ds\\
&=2\Re\int_0^t\int_0^s\<N(u(s)),e^{i(s-s')\Delta}N(u(s'))\>ds'ds\\
&=-2\Im \int_0^t\left\langle N(u(s)),u(s)+i\int_0^se^{i(s-s')\Delta}N(u(s'))ds'\right\rangle ds\\
&=2\Im \int_0^t\left\langle e^{is\Delta}u_0,N(u(s))\right\rangle ds,
\end{align*}
where, in the second line, we have written the integral over $[0,t]\times[0,t]$ as $$\int_0^t\int_0^s(\cdots)ds'ds+\int_0^t\int_0^{s'}(\cdots)dsds'=2\int_0^t\int_0^s(\cdots)ds'ds$$ by Fubini's theorem, and we have used the fact $\Im \<N(u),u\>=0$ and \eqref{proposition_appendix_B_2} in the third and the last lines, respectively. Hence the second and the last terms  in the right hand side of \eqref{proposition_appendix_B_2} cancel each other out, yielding the desired $L^2$-conservation law \eqref{proposition_appendix_B_1}. 
\end{proof}

%remark
\begin{remark}
\label{remark_appendix_B_2}
In addition to the assumption in Proposition \ref{proposition_B_1}, we suppose $\beta<d/2$. Note that this condition is always satisfied under \eqref{long_range_INLS}. Then $q_3<2(\alpha+1)$ and hence $L^{q_3}\subset L^{q_3,2(\alpha+1)}$. In particular, we obtain
$$
\norm{|x|^{-\beta}|v|^{\alpha+1}}_{L^{p_3'}_TL^{q_3'}_x}\lesssim \norm{|x|^{-\beta}|v|^{\alpha+1}}_{L^{p_3'}_TL^{q_3',2}_x}\lesssim \norm{v}_{L^{p_3'(\alpha+1)}_TL^{q_3,2(\alpha+1)}_x}^{\alpha+1}\lesssim 
T^\delta\norm{v}_{L^{p_3}_TL^{q_3}_x}^{\alpha+1}
$$
from which one can see that the above proof also works if $L^{p_3}_TL^{q_3,2}_x$ replaced by $L^{p_3}_TL^{q_3}_x$. Therefore, in such a case, Proposition \ref{proposition_B_1} holds with $L^{p_3}_{\loc}(\R;L^{q_3,2}(\R^d))$ replaced by $L^{p_3}_{\loc}(\R;L^{q_3}(\R^d))$. 
\end{remark}

%References

\end{document}